\documentclass[a4, 12pt]{amsart}
\usepackage{amssymb,amstext,amsmath,amscd,amsthm,amsfonts,enumerate,latexsym}

\usepackage[all]{xy}
\usepackage{mathrsfs}

\newtheorem{thm}{Theorem}[section]
\newtheorem{prop}[thm]{Proposition}
\newtheorem{lem}[thm]{Lemma}
\newtheorem{cor}[thm]{Corollary}

\newtheorem{example}[thm]{Example}

\theoremstyle{definition}

\newtheorem*{claim*}{Claim}
\newtheorem*{note*}{Note}

\newtheorem{rem}[thm]{Remark}

\numberwithin{equation}{thm}
\def\Hom{\mathrm{Hom}}

\def\Ext{\mathrm{Ext}}

\newcommand{\coker}{\mathop{\mathrm{coker}}\nolimits}

\newcommand{\depth}{\mathop{\mathrm{depth}}\nolimits}
\newcommand{\Ass}{\mathop{\mathrm{Ass}}\nolimits}

\newcommand{\Min}{\mathop{\mathrm{Min}}\nolimits}

\def\ai{\mathrm{a}}
\def\fka{\mathfrak a}

\def\E{\mathrm E}
\def\m{\mathfrak m}
\def\n{\mathfrak n}

\def\p{\mathfrak p}
\def\q{\mathfrak q}

\def\L{\mathcal L}
\def\w{\omega}
\def\P{\mathcal P}

\def\Z{\Bbb Z}

\def\V{\mathrm V}

\newcommand{\calL}{\mathcal{L}}
\newcommand{\calK}{\mathcal{K}}
\newcommand{\calM}{\mathcal{M}}
\newcommand{\calN}{\mathcal{N}}

\newcommand{\NCM}{\mathrm{NCM}}

\def\K{\mathrm{K}}
\def\H{\mathrm{H}}
\def\E{\mathrm{E}}

\def\fkM{\mathfrak M}
\def\fkN{\mathfrak N}

\def\th{\mbox{\tiny th}}

\def\Supp{\mathrm{Supp}}

\def\height{\mathrm{ht}}

\newcommand{\Spec}{\mathop{\mathrm{Spec}}\nolimits}
\newcommand{\Proj}{\mathop{\mathrm{Proj}}\nolimits}
\newcommand{\grade}{\mathop{\mathrm{grade}}\nolimits}

\def\A{{\mathcal A}}

\def\G{{\mathcal G}}
\def\R{{\mathcal R}}
\def\F{{\mathcal F}}

\def\M{{\mathcal M}}

\begin{document}


\title{Some characterizations of Gorenstein Rees Algebras}

\author{Shin-ichiro Iai} 
\address{Mathematics laboratory, Sapporo College, Hokkaido University of Education, 002-8502 JAPAN} 
\email{iai.shinichiro@s.hokkyodai.ac.jp}


\maketitle

\begin{center}
\small{
\it{Dedicated to the memory of Professor Shiro Goto and Professor Yasuhiro Shimoda} }
\end{center}

\begin{abstract}
The aim of this paper is to elucidate the relationship between the Gorenstein Rees algebra $\R(I):=\bigoplus_{i\ge 0}I^i$ of an ideal $I$ in a complete Noetherian local ring $A$ and the graded canonical module of the extended Rees algebra $\R'(I):=\bigoplus_{i\in\Z}I^i$. It is known that the Gorensteinness of $\R(I)$ is closely related to the property of the graded canonical module of the associated graded ring $\G(I):=\bigoplus_{i\ge 0}I^i/I^{i+1}$. However, there appears to be a shortage of satisfactory references analyzing the relationship between $\R(I)$ and $\R'(I)$ unless the ring $\G(I)$ is Cohen-Macaulay. This paper provides a characterization of the Gorenstein property of $\R(I)$ using the graded canonical module of $\R'(I)$ without assuming that the base ring $A$ is Cohen-Macaulay. Applying our criterion, we demonstrate that a certain Kawasaki's arithmetic Cohen-Macaulayfication becomes a Gorenstein ring when $A$ is a quasi-Gorenstein local ring with finite local cohomology.
\end{abstract}

\section{Introduction.}

Throughout this paper, all rings are assumed to be commutative with nonzero identity. 
Let $A$ be a Noetherian local ring
with maximal ideal $\m$ 
of dimension $d$ 
and let $A[\,t\,]$ denote a graded polynomial ring over $A$ in one variable $t$ such that $\deg t=1$ and $\deg a=0$ for all $a\in A\setminus\{0 \}$.
Let $\F=\{F_i\}_{i\in\mathbb{Z}}$ be a filtration of ideals in $A$ (i.e., $\F$ is a family of ideals $F_i$ in $A$ indexed by $\mathbb{Z}$ such that $F_i\supseteq F_{i+1}$, $ F_i F_j\subseteq F_{i+j}$ for all $i,j\in\mathbb{Z}$, and $F_0=A$).  Then we consider the following three graded $A$-algebras:
\begin{enumerate}
\item[] $\R(\F)=\sum_{i\ge 0}F_it^i\ \big(\hspace{-0.6mm}\subseteq A[\,t\,]\,\big)$,
\item[] $\R'(\F)=\sum_{i\in\mathbb{Z}}F_it^i\ \big(\hspace{-0.6mm}\subseteq A[\,t, t^{-1}\,]\,\big)$, and
\item[] $\G(\F)=\R'(\F)/t^{-1}\R'(\F)$
\end{enumerate}
and call them respectively the Rees algebra, the extended Rees algebra, and the associated graded ring of $\F$ (see, e.g, \cite[4.5 Filtered rings]{BH} for details). In the case where $\F=\{I^i\}_{i\in\mathbb{Z}}$ for some ideal $I$ in $A$, we shall utilize the alternative notation $\R(I)$, $\R'(I)$, and $\G(I)$. This paper will discuss the relationship between the Gorensteinness of the Rees algebra $\R(\F)$ and the graded canonical module of the extended Rees algebra $\R'(\F)$. 

The relationship of ring-theoretic properties, particularly Cohen-Macaulayness and Gorensteinness, between $\R(\F)$ and $\G(\F)$ has been thoroughly investigated.
For instance, 
in \cite{GS}, Goto and Shimoda showed that the Rees algebra $\R(\m)$ is Cohen-Macaulay (resp. Gorenstein) if and only if the associated graded ring $\G(\m)$ is a Cohen-Macaulay ring and $\ai(\G(\m))<0$ (resp. a Gorenstein ring and $\ai(\G(\m))=-2$), provided the ring $A$ is a Cohen-Macaulay ring with $d\ge 1$ (resp. a Gorenstein ring with $d\ge 2$), where $\ai(\G(\m))$ denotes the $a$-invariant of $\G(\m)$ (see \cite[Definition (3.1.4)]{GW}). Respecting this monumental result, a characterization in the relationship between ring-theoretic properties of the Rees algebra and those of the associated graded ring together with its $a$-invariant is referred to as a characterization {\it of Goto-Shimoda type}.

In \cite{H82}, Huneke showed that Cohen-Macaulayness (resp. Gorensteinness) of the Rees algebra $\R(I)$ yields Cohen-Macaulayness (resp. Gorensteinness) of
the associated graded ring $\G(I)$, provided the ring $A$ is Cohen-Macaulay (resp. the ring $A$ is Gorenstein) and $I$ is an ideal of $A$ with $\height_AI\ge 1$ (resp. $\height_AI\ge 2$). 
The observation in \cite{H82} presents an effective basis for connecting $\R(I)$ and $\G(I)$ through the two short exact sequences
\begin{center}
$
0\to It\hspace{-1mm}\cdot\hspace{-1mm}\R(I)\to\R(I)\to A\to 0 $\quad and\quad $ 0\to It\hspace{-1mm}\cdot\hspace{-1mm}\R(I)(1)\to\R(I)\to \G(I)\to 0
$
\end{center}
of graded $\R(I)$-modules, which have significantly influenced subsequent research.
Their achievements in \cite{GS} and \cite{H82} have become landmarks in the research on Rees algebras and associated graded rings, leading to subsequent publications of characterizations of Goto-Shimoda type.

The results in \cite{GS} and \cite{H82} assume that the base ring $A$ is Cohen-Macaulay, but in general, the Cohen-Macaulay property of the Rees algebra $\R(\F)$ does not necessarily descend to the base ring $A$ (see \cite[Example 2.2]{HR}). 
For research aimed at weakening the assumption that the base ring $A$ is Cohen-Macaulay, Trung and Ikeda provided an outstanding characterization of Goto-Shimoda type for the Cohen-Macaulayness of the Rees algebra of an arbitrary ideal with positive height (see \cite[Theorem 1.1]{TI}). Their result has made the Cohen-Macaulay property of the Rees algebra controllable in many cases, even when the base ring $A$ is not Cohen-Macaulay. We remark that Trung-Ikeda's theorem was enhanced to the case of the Rees algebra of a filtration $\F$ (see \cite[Theorem (1.1)]{GN} and \cite[Theorem 1.1]{Vi}).

For the Gorensteinness of Rees algebras, which is a central theme of the present paper, a landmark result was discovered by Ikeda (\cite[Theorem 3.1]{I}). 
Let us here state his result. Ikeda shows that the following two conditions are equivalent, assuming that $I$ is an ideal of $A$ with $\grade I\ge 2$ and the ring $\R(I)$ is Cohen-Macaulay.
\begin{enumerate}
	\item $\R(I)$ is a Gorenstein ring. \vspace{2mm}
	\item
\begin{enumerate}
\item[{\rm (i)}] Both the rings $A$ and $\G(I)$ are quasi-Gorenstein and 
\item[{\rm (ii)}] $\ai(G(I))=-2$.
\end{enumerate}
\end{enumerate}
\vspace{0mm}
Here, a Noetherian local ring $B$ is said to be {\it quasi-Gorenstein} if $B$ has the canonical module $\K_B$ that is a free $B$-module of rank one. Furthermore, a Noetherian $\Z$-graded ring $S$ with a unique homogeneous maximal ideal $\fkN$ is called {\it quasi-Gorenstein} if $S$ has the graded canonical module $\K_S$ and the local ring $S_\fkN$ is quasi-Gorenstein.
This is equivalent to saying that $S$ has the graded canonical module $\K_S$ such that $\K_S\cong S(k)$ as a graded $S$-module for some $k\in\Z$. We refer the reader to \cite{A}, \cite{BH}, \cite{HIO}, \cite{HK}, \cite{I}, and \cite{S1} for details of the (graded) canonical modules. 

Regarding the Cohen-Macaulayness of the Rees algebra $\R(I)$, which is a necessary condition for its Gorensteinness, we have Trung-Ikeda's criterion as mentioned above. Hence, Ikeda's theorem asserts that the Gorenstein property of $\R(I)$ can be characterized using only information from the rings $A$ and $\G(I)$ (in reality, Trung-Ikeda's criterion was developed after Ikeda's theorem was announced).
Ikeda's theorem was also enhanced to the case of a filtration $\F$ (see \cite[Theorem (1.3)]{GN} and \cite[Theorem 3.1]{TVZ}). The present paper aims to make the enhanced version of Ikeda's theorem as provided by \cite{GN} and \cite{TVZ}, more applicable in certain cases.
To explain the details, let us establish the following notation.
In this section, we assume that $\F$ is a filtration of ideals such that $F_1\subsetneq F_0$ and $\R(\F)$ is a Noetherian ring with $\dim \R(\F)=d+1$. 
Hence, both the rings $\R'(\F)$ and $\G(\F)$ are Noetherian, as $\R'(\F)=\R(\F)[t^{-1}]$ and $\G(\F)=\R'(\F)/t^{-1}\R'(\F)$.
We furthermore assume that the ring $A$ is a homomorphic image of a Gorenstein local ring and denote by $\K_{\R(\F)}$, $\K_{\R'(\F)}$, and $\K_{\G(\F)}$ the graded canonical module of $\R(\F)$, $\R'(\F)$, and $\G(\F)$, respectively (see, e.g., \cite[(1.6)]{I} for the definition of the graded canonical module). 
For each Noetherian $\Z$-graded ring $S$ and for each graded $S$-module $E$, the $i^{\th}$ homogeneous component of $E$ is written by $E_i$, and hence
$S=\bigoplus_{i\in\Z}S_i$ and
$E=\bigoplus_{i\in\Z}E_i$. 
For each integer $n$, we denote by $E(n)$ the graded $S$-module whose graduation is given by $[E(n)]_i=E_{n+i}$ for all $i\in\mathbb{Z}$, and we define
\begin{center}
$E_{\ge n}=\bigoplus_{i\ge n}E_i$,
\end{center}
in particular, we have $\R'(\F)_{\ge 0}=\R(\F)$.
The notation $E_+$ is specially used to denote $E_{\ge 1}$.
It should be noted that $S_{\ge 0}$ is a graded subring of $S$ and $E_{\ge n}$ is a graded $S_{\ge 0}$-submodule of $E$. 
We recall
$$
\ai(\G(\F))=-\min\{i\in\Z\mid [\K_{\G(\F)}]_i\neq 0\},
$$
which is called the $a$-invariant of $\G(\F)$. 

Let us introduce some results stated in \cite{GN} and \cite{TVZ} for comparison with our main theorem. 
In \cite{GN}, Goto and Nishida did not only extend Ikeda's theorem to the case of filtration $\F$ of ideals, but also showed that under a mild assumption, the Gorenstein property of the ring $\R(\F)$ gives rise to 
the remarkable short exact sequence
$$
0\to \G(\F)(-2)\to \K_{\G(\F)}\to \Ext_A^1(A/F_1, A)(-1)\to 0
$$
of graded $\G(\F)$-modules (see \cite[Theorems (1.3) and (4.2)]{GN}).
Meanwhile, in \cite{TVZ}, Trung, Vi\^et, and Zarzuela succeeded in showing that the following two conditions are equivalent, only assuming that $\R(\F)$ is a Cohen-Macaulay ring (\cite[Theorem 1.1]{TVZ}).

\vspace{1mm}
\begin{enumerate}
	\item $\R(\F)$ is a Gorenstein ring. \vspace{2mm}
	\item $[\K_{\G(\F)}]_{\ge 2}\cong\G(\F)(-2)$ as a graded $\G(\F)$-module.
\if0\begin{enumerate}
\item $\R(\F)$ is a Cohen-Macaulay ring and
\item $[\K_{\G(\F)}]_{\ge 2}\cong\G(\F)(-2)$ as a graded $\G(\F)$-module.
\end{enumerate}\fi
\end{enumerate}
\vspace{1mm}
This  Trung-Vi\^et-Zarzuela's theorem represents a significant improvement over Ikeda's theorem
because it is applicable even when $\operatorname{grade} I=1$, while Ikeda's theorem assumed that $\operatorname{grade} I\geq 2$.   
It stands as the most optimal outcome in the realm of Goto-Shimoda type results for Gorenstein Rees algebras.

Although we have obtained satisfactory results concerning the relationship between $\R(\F)$ and $\G(\F)$ as mentioned above, there appears to be a shortage of satisfactory references analyzing the relationship between $\R(\F)$ and $\R'(\F)$ unless the ring $\R'(\F)$ is Cohen-Macaulay.
The purpose of the present paper is to reinforce their results stated in \cite{GN} and \cite{TVZ} by incorporating the terms of the canonical module $\K_{\R'(\F)}$ of the extended Rees algebra $\R'(\F)$. Let us now state the main result in this paper as follows.

\begin{thm}\label{main1}
Assume that $\R(\F)$ is a Cohen-Macaulay ring. Then the following three conditions are equivalent.
\begin{enumerate}
	\item $\R(\F)$ is a Gorenstein ring. \vspace{2mm}
	\item $[\K_{\R'(\F)}]_{\ge 1}\cong\R(\F)(-1)$ as a graded $\R(\F)$-module.\vspace{2mm}
	\item 
\begin{enumerate}
\item[{\rm (i)}] $[\K_{\R'(\F)}]_1\cong A$ as an $A$-module and 
\item[{\rm (ii)}] there exists a monomorphism $\G(\F)(-2)\hookrightarrow \K_{\G(\F)}$ of graded $\G(\F)$-modules. 
\end{enumerate}
\end{enumerate}
When this is the case, $[\K_{\R'(\F)}]_i\cong\K_A$ as an $A$-module for all integers $i\le 0$.
\end{thm}

The equivalence $(1)\Leftrightarrow(2)$ in Theorem \ref{main1} aims to reinforce the result of Trung, Viet, and Zarzuela by substituting $\K_{\G(\F)}$ with $\K_{\R'(\F)}$.

The reason why there are so few results analyzing the relationship between $\R(\F)$ and $\R'(\F)$
is speculated to be due to the lack of establishment of the similarity of two Huneke's short exact sequences as mentioned above in the category of graded $\R'(\F)$-modules.
Of course, since $\G(\F)=\R'(\F)/t^{-1}\R'(\F)$, we have the short exact sequence
$$0\to \R'(\F)(1) \xrightarrow{t^{-1}}\R'(\F) \to \G(\F)\to 0$$
of graded $\R'(\F)$-modules induced by multiplication by $t^{-1}$ on $\R'(\F)$, but, with only this, the information in $\R'(\F)$ and $\G(\F)$ does not completely correspond when $\R'(\F)$ is not Cohen-Macaulay. For instance, it is routine to check that the short exact sequence above induces the monomorphism 
$$\left(\K_{\R'(\F)}\big/t^{-1}\K_{\R'(\F)}\right)(-1)\hookrightarrow\K_{\G(\F)}
$$ 
of graded $\G(\F)$-modules, but it is not certain whether it is bijective.
It is well-known that, assuming the ring $\R'(\F)$ is Cohen-Macaulay, 
we obtain that it is bijective
(see, e.g., \cite[Corollary 3.6.14]{BH}).
However, such an assumption does not seem to conform with the reinforcement of Trung-Viet-Zarzuela's result mentioned above, because there is a wide gap between the properties of $\R(\F)$ and $\R'(\F)$. For instance, unlike $\R(\F)$, the Cohen-Macaulayness of $\R'(\F)$ can descend to the base ring $A$, since $$\R'(\F)/(t^{-1}-1) \cong A$$ as a ring.
It is preferable not to impose unnecessary restrictions on the base ring $A$.
Theorem \ref{main1} in the present paper does not require $A$ to be a Cohen-Macaulay ring.
To prove Theorem \ref{main1}, we will show that the monomorphism stated above is bijective, provided the Rees algebra $\R(\F)$ is a Cohen-Macaulay ring (see Corollary \ref{maincor}).

Besides, the equivalence $(1)\Leftrightarrow(3)$ in Theorem \ref{main1} seems to render some results in \cite{GN} and \cite{TVZ} more applicable in certain cases.
For example, see the equivalence $(1)\Leftrightarrow(3)$ in the corollary below, which follows as a consequence of Theorem \ref{main1}.

\begin{cor}\label{maincor1} 
Assume that $\R(\F)$ is a Cohen-Macaulay ring with $\grade F_1\ge 2$. Then the following three conditions are equivalent.\vspace{0mm}
\begin{enumerate}
	\item $\R(\F)$ is a Gorenstein ring.
\vspace{2mm}
	\item $\K_{\R'(\F)}\cong\R'(\F)(-1)$ as a graded $\R'(\F)$-module.
\vspace{2mm}
	\item
\begin{enumerate}
\item[{\rm (i)}] $\K_A\cong A$ as an $A$-module,  
\item[{\rm (ii)}] $\ai(\G(\F))=-2$, and 
\item[{\rm (iii)}] there exists a monomorphism $\G(\F)(-2)\hookrightarrow \K_{\G(\F)}$ of graded $\G(\F)$-modules. 
\end{enumerate}
\end{enumerate}
\end{cor}

The equivalence $(1)\Leftrightarrow(2)$ in Corollary \ref{maincor1} elucidates the relationship between the Gorenstein property of $\R(\F)$ and a quasi-Gorenstein property of $\R'(\F)$.
The equivalence $(1)\Leftrightarrow(3)$
generalizes the result \cite[Theorem 1.3]{GN} due to Goto and Nishida.

In order to state the next result, let us fix the following notation.
Let $\mathrm{V}(F_1)$ denote the set of prime ideals $\p$ in $A$ such that $F_1\subseteq \p$.
For each $\p\in\mathrm{V}(F_1)$,
we put $\F_\p=\{F_i A_\p\}_{i\in\mathbb{Z}}$, which is a filtration of ideals of $A_\p$, and set $\ell(\F_\p)=\dim\G(\F_\p)/\p\,\G(\F_\p)$, which is called the analytic spread of $\F_\p$. 
Put $$\A(\F)=\bigcup_{i\ge 1}\Ass_AA/F_i.$$ It is known that $\A(\F)=
\Ass_A\G(\F)=\{P\cap A\mid P\in\Ass\G(\F)\}$, which is a finite set (cf. \cite[(3)]{B} and \cite[Lemma 3]{ME}, and see also \cite[Exercise 6.7]{M}), 
and moreover the equality $$\A(\F)=\{\p\in\mathrm{V}(F_1)\mid \dim A_\p=\ell(\F_\p)\}$$ holds when the ring $A$ is quasi-unmixed and 
$\Ass \G(\F)=\Min \G(\F)$
(cf. \cite[Proposition 4.1]{Mc1} and see also \cite[Proof of Lemma 3.1]{GI}). Here, $\Min \G(\F)$ denotes the set of all minimal prime ideals of the ring $\G(\F)$. 
With this notation, we present the following result.

\begin{thm}\label{main2}
Let $k$ be an integer. Assume that $\K_A\cong A$ as an $A$-module. Then the following two conditions are equivalent.\vspace{1mm}
\begin{enumerate}
	\item $\K_{\R'(\F)}\cong\R'(\F)(k)$ as a graded $\R'(\F)$-module.\vspace{2mm}
	\item $\K_{\R'(\F_\p)}\cong\R'(\F_\p)(k)$ as a graded $\R'(\F_\p)$-module for every $\p\in\A(\F)$.
\end{enumerate}
\end{thm}

It should be noted that the condition where $\K_{\R'(\F)}\cong\R'(\F)(k)$ as a graded $\R'(\F)$-module for some integer $k$ implies that $\K_A\cong A$ as an $A$-module.

As a consequence of Theorem \ref{main2} together with Corollary \ref{maincor1},
we obtain the following.

\begin{cor}\label{maincor2}
Assume that $\R(\F)$ is a Cohen-Macaulay ring with $\grade F_1\ge 2$ and
suppose that $\K_A\cong A$ as an $A$-module.
Then the following two conditions are equivalent.\vspace{1mm}
\begin{enumerate}
	\item $\R(\F)$ is a Gorenstein ring.\vspace{2mm}
	\item $\R(\F_\p)$ is a Gorenstein ring for every $\p\in\A(\F)$.
\end{enumerate}
\end{cor}

We note that while a result akin to Corollary \ref{maincor2} is presented in \cite[Proposition 4.1]{HHK1}, the result stated here does not require the base ring $A$ to be Cohen-Macaulay.

In the rest of this section, we will focus on the Gorensteinness of the Rees algebra $\R(I^n)$ of powers of an ideal $I$ for some positive integer $n$.
The relationship between $\R(I^n)$ and $\G(I)$ is profoundly explored (see, e.g., \cite[Theorem (2.3)]{HRS},
\cite[Theorem (3.5)]{HRZ},
\cite[Theorem 3.12 and Remark 4.9]{HHR}, and
\cite[Theorem 3.1]{HHK1}). 
These results relies on a mild assumption about the ring $\R(I)$. For instance, it is assumed that $\R(I)$ is Cohen-Macaulay or, more generally, that $[\H^{d}_{\fkM}(\R(I))]_i=0$ for all integers $i\le -n$. Here, $\H^{d}_{\fkM}(\R(I))$ denotes the $d^{\th}$ graded local cohomology module of $\R(I)$ with respect to the homogeneous maximal ideal $\fkM$ in $\R(I)$. Our subsequent findings regarding the connection between $\R(I^n)$ and $\R'(I)$ do not require such an assumption about the ring $\R(I)$.
To state our result, let us set up some notation.
For each regular ideal $J$ of $A$ (i.e., $J$ contains a nonzero divisor on $A$), we define\vspace{-1mm} 
$$J^\ast=\bigcup_{i\ge 0}J^{i+1}:_AJ^i\vspace{-1mm},$$
which is called the Ratliff-Rush closure of $J$, given by \cite{RR}.
Let $I$ be a regular ideal in $A$ such that $I\neq A$. We set $I^i=A$ for every integer $i\le 0$. 
In what follows, we consider the case where $\F=\{I^i\}_{i\in\mathbb{Z}}$, and then we define $$\F^\ast=\{(I^i)^\ast\}_{i\in\mathbb{Z}},$$ which forms a filtration of ideals in $A$. With this notation, we can state the following theorem, which clarifies the relationship between the Gorenstein Rees algebra $\R(I^n)$ of powers $I^n$ of $I$ and the graded canonical module $\K_{\R'(I)}$ of $\R'(I)$ of an ordinary ideal $I$.

\begin{thm}\label{in}
Let $n$ be an positive integer and assume that $\R(I^n)$ is a Cohen-Macaulay ring. Then the following three conditions are equivalent.
\begin{enumerate}
\item $\R(I^n)$ is a Gorenstein ring.\vspace{2mm}
\item $[\K_{\R'(I)}]_{\ge n}\cong\R(\F^\ast)(-n)$ as a graded $\R(I)$-module.
\vspace{2mm}
\item $[\K_{\R'(\F^\ast)}]_{\ge n}\cong\R(\F^\ast)(-n)$ as a graded $\R(\F^\ast)$-module..\vspace{1mm}
\end{enumerate}
When this is the case, there exists a monomorphism $\G(\F^\ast)(-n-1)\hookrightarrow \K_{\G(\F^\ast)}$ of graded $\G(\F^\ast)$-modules.
\end{thm}

For the last assertion in Theorem \ref{in}, we emphasize that Herrmann, Hyry, and Ribbe are the first to reveal that the Gorenstein Rees algebra $\R(I^n)$ provides an embedding of $\G(\F^\ast)(-n-1)$ into the canonical module $\K_{\G(I)}$ of graded $\G(I)$-modules (cf. \cite[Theorem 3.9 (3)]{HHR}, and see also Remark \ref{injast}). Hence, the last assertion in Theorem \ref{in} is a slight modification of their result.

When $\grade I\ge 2$, it is observed that the Gorenstein property of $\R(I^n)$ is equivalent to saying that the ring $\R'(\F^\ast)$ is quasi-Gorenstein, as follows.

\begin{cor}\label{incor}
Suppose that $\grade I\ge 2$.
Assume that $\R(I^n)$ is a Cohen-Macaulay ring for some integer $n>0$. Then the following two conditions are equivalent.
\begin{enumerate}
\item $\R(I^n)$ is a Gorenstein ring.\vspace{2mm}
\item $\K_{\R'(I)}\cong\R'(\F^\ast)(-n)$ as a graded $\R'(I)$-module.
\vspace{2mm} \item $\K_{\R'(\F^\ast)}\cong\R'(\F^\ast)(-n)$ as a graded $\R'(\F^\ast)$-module.\vspace{1mm}
\end{enumerate}
When this is the case, the equality $\mathrm{a}(\G(I))=-n-1$ holds true.
\end{cor}

The last assertion in Corollary \ref{incor} is a generalization of the result due to
\cite[Corollary 3.10]{HHR} in the case where $\grade I\ge 2$. 
Their theorem relies on the assumption that $\H_\fkM^d(\R(I))=0$ and their proof does not work without this assumption.

As an illustration of the aforementioned results, let us consider the following example, which addresses the existence of an ideal $J$ in $A$ such that the Rees algebra $\R(J)$ is Gorenstein.
To state it, let us set up the following notation. 
In the rest of this section, 
assume that $d\ge 3$ and $\dim A/\p=d$ for every $\p\in\Ass A$. 
Put 
$\fka (A)=\prod_{i=0}^{d-1} 0:_A\H_\m^i(A)$, where $\H_\m^i(A)$ denotes the $i^{\th}$ local cohomology modules of $A$ with respect to $\m$, and set
$\NCM (A)=\{\p\in\Spec A\mid A_\p$ is not Cohen-Macaulay$\}$.
Then $\NCM (A)=\mathrm{V}(\fka (A))$ by \cite[Theorem 3]{S}. Suppose that $\dim \NCM (A)\le 1$ and there exists a system $x_1, x_2,\dots , x_d$ of parameters for $A$ such that $x_2, x_3, \dots, x_d \in \fka (A)$, and we consider the ideal\vspace{0mm} $$I=\bigcup_{i\ge 1}(x_2, x_3, \dots, x_d):x_1^i.\vspace{0mm}$$ Then, Kawasaki showed that the Rees algebra $\R(I^n)$ is Cohen-Macaulay for every integer $n\ge d-2$ in \cite[Theorems 2.4 and 3.1]{K2}. 
With this notation, the final result in this section can be stated as follows.

\begin{example}[cf. \cite{K2}]\label{qGex}
Assume that $\K_A\cong A$ as an $A$-module and $\dim \NCM (A)\le 0$. Then the following three assertions hold true.
\begin{enumerate}
	\item $\grade \G(I)_+>0$, and hence ${(I^i)}^\ast=I^i$ for all integers $i\in\Z$. \vspace{1mm}
	\item $\ai(\G(I))=-d+1$ and $\K_{\R'(I)}\cong\R'(I)(-d+2)$ as a graded $\R'(I)$-module.
Hence, choosing two positive integers $m$ and $n$ such that $mn=d-2$, one obtains that $\K_{\R'(I^n)}\cong\R'(I^n)(-m)$ as a graded $\R'(I^n)$-module. \vspace{1mm}
	\item Let $n$ be a positive integer.
Then,
\begin{center}
$\R(I^n)$ is a Gorenstein ring $\iff n= d-2.$\hspace{2.3cm}\,
\end{center}
\end{enumerate}
\end{example}

When the ring $A$ is generalized Cohen-Macaulay such that $\K_A\cong A$ as an $A$-module, the assertion (2) above demonstrates that Kawasaki's ideal $I$ leads to a quasi-Gorenstein extended Rees algebra $\R'(I)$. Moreover, the assertion (3) illustrates the existence of an ideal $J$ in $A$ such that the Rees algebra $\R(J)$ is Gorenstein.

To conclude this section, let us explain how to organize this paper.
The proofs of Theorems \ref{main1} and \ref{main2}, as well as Corollaries \ref{maincor1} and \ref{maincor2}, will be provided in Section 3. Section 2 summarizes some preliminary results on $\F$-filtrations, which we need to prove the aforementioned results. Section 4 is  devoted to investigating the Rees algebras of powers of ideals and gives the proofs of Theorem \ref{in}, Corollary \ref{incor}, and Example \ref{qGex}.

\section{Preliminary results on $\F$-filtrations}

Throughout this section, let $A$ be a commutative ring and $\F=\{F_n\}_{n\in\Z}$ a filtration of ideals of $A$ (i.e., $\F$ is a family of ideals $F_i$ in $A$ indexed by $\mathbb{Z}$ such that $F_i\supseteq F_{i+1}$, $ F_i F_j\subseteq F_{i+j}$ for all $i,j\in\mathbb{Z}$, and $F_0=A$). We consider the rings $\R(\F)$, $\R'(\F)$, and $\G(\F)$ as in Section 1. 
Let $M$ be an $A$-module and
let $\M=\{M_i\}_{i\in\Z}$ stand for a family of $A$-submodules of $M$ indexed by $\Z$. We say that
$\M$ is {\it an $\F$-filtration of $M$}, if the following three conditions are satisfied:
\begin{enumerate}
	\item [(i)] $M_{i} \supseteq M_{i+1}$ for all $i\in \Z$.
	\item [(ii)] $F_iM_j \subseteq M_{i+j}$ for all $i, j \in \Z$.
	\item [(iii)] $M_i=M$ for all integers $i\ll 0$.
\end{enumerate}

This section will summarize some preliminary results on $\F$-filtrations, which we need to prove the results in Section 1.
Let $t$ be an indeterminate over $A$. We denote by $A[t,t^{-1}]$ the graded Laurent polynomial ring in one variable $t$ of $\deg t=1$  and $\deg a=0$ for all $a\in A\setminus\{0 \}$.
For each $\F$-filtration $\M$ of $M$, we define
$$
\R'(\M) = \sum_{i\in\Z}At^i\otimes_A M_i\ \big(\subseteq A[t,t^{-1}]\otimes_A M
\big),$$
which is a graded $\R'(\F)$-submodule of $A[t,t^{-1}]\otimes_A M$,
and call it the extended Rees module of $\M$. 
Furthermore, we put 
$$
\G(\M) = \R'(\M)/t^{-1}\R'(\M),
$$
which is a graded $\G(\F)$-module,
and call it the associated graded module of $\M$. 
We notice that, for each $n\in\Z$, $\R'(\M)_{\ge n}$ is a graded $\R(\F)$-submodule of $\R'(\M)$. We specially set $\R(\M)=\R'(\M)_{\ge 0}$ and call it the Rees module of $\M$.
Put $$T=\{(t^{-1})^{n}\mid n\in\Z,\ n\ge 0\},$$ which is the multiplicatively closed subset of $\R'(\F)$.
Let us begin with the following.

\begin{lem}\label{filtration}
Let $W$ be a finitely generated $A$-module and let $\calK$ be a graded $\R'(\F)$-module such that $t^{-1}$ is a nonzero divisor on $\calK$. Assume that $T^{-1}\calK\cong A[t,t^{-1}]\otimes_A W$ as a graded $T^{-1}\R'(\F)$-module.
Then there exists an $\F$-filtration $$\w=\{\w_i\}_{i\in\Z}$$ of $W$ such that $\calK\cong \R'(\w)$ as a graded $\R'(\F)$-module.    
\end{lem}

\begin{proof}
Since $t^{-1}$ is a nonzero divisor on $\calK$, the canonical homomorphism $\calK\to T^{-1}\calK$ of localization with respect to $T$ is injective.
By the composition of this injection and the given isomorphism $T^{-1}\calK\cong A[t,t^{-1}]\otimes_A W$, we get the monomorphism $$\varphi: \calK\hookrightarrow A[t,t^{-1}]\otimes_AW$$
of graded $\R'(\F)$-modules.
Put $$\w_i=\{x\in W\mid t^i\otimes_A x\in \varphi (\calK)\}$$
for each integer $i$. It is routine to check that each $\w_i$ is an $A$-submodule of $W$ for each $i\in\Z$, and if $N$ is an $A$-submodule of $W$ such that $At^{i}\otimes_A N\subseteq \varphi (\calK)$, then $N\subseteq \w_{i}$ by the definition of $w_i$.

Let us check that 
$\w=\{\w_i\}_{i\in\Z}$ is an $\F$-filtration of $\K_M$ such that $\calK\cong\R'(\w)$ as a graded $\R'(\F)$-module.
Let $i$ and $j$ be integers.
From $At^{i+1}\otimes_A\w_{i+1}\subseteq \varphi (\calK)$, it follows that $t^{-1}\cdot (At^{i+1}\otimes_A\w_{i+1})\subseteq \varphi (\calK)$, whence
$
At^{i}\otimes_A\w_{i+1}\subseteq \varphi (\calK)
,$ and thus
$\w_{i+1}\subseteq\w_i$.
Since $At^{i+j}\otimes_AF_iw_j=F_it^i\cdot (At^j\otimes_A w_j)\subseteq \varphi (\calK),$
we get $F_iw_j\subseteq \w_{i+j}$. We obtain from the equality $\varphi (\calK)=\R'(\w)$ that $\calK\cong \R'(\w)$ as a graded $\R(\F)$-module, as required.

Finally, we will show that $\w_i=W$ for all $i\ll 0$. Take any element $x\in W$. Then $1\otimes x\in A[t,t^{-1}]\otimes_A W$. 
We can find a nonnegative integer $n$ such that $(t^{-1})^n(1\otimes x)=\varphi (\xi)$ for some $\xi\in \calK$. 
Hence $t^{-n}\otimes x\in \varphi (\calK)$, so that $x\in\w_{-n}$. Since $W$ is a finitely generated $A$-module, we see $\w_k=W$ for some integer $k$. This completes the proof.
\end{proof}

Let $\mathcal{M}=\{M_i\}_{i\in\Z}$ be an $\F$-filtration of $M$ and let
$\calN=\{N_i\}_{i\in\Z}$ be an $\F$-filtration  of an $A$-module $N$. For each homomorphism $\varphi :\R'(\calM)\to\R'(\calN)$ of graded $\R'(\F)$-modules and each integer $i$, we denote by $$\varphi_i: A_it^i\otimes M_i\to A_it^i\otimes N_i$$ the $i^{\th}$ homogeneous component of $\varphi$. 
With particular attention to detail, we define $$\varphi |_i : M_i\to N_i$$ as the homomorphism  of $A$-modules given by $\varphi_i(t^i\otimes x)=t^i\otimes \varphi |_i(x)$ for all $x\in M_i$, namely, $\varphi |_i$ is the $A$-homomorphism such that the diagram
\[
\xymatrix
{
A_it^i\otimes M_i \ar[r]^{\varphi_i}& A_it^i\otimes N_i \\
M_i \ar[r]^{\varphi |_i}  \ar[u]_{\wr} & N_i  \ar[u]_{\wr} 
}
\] 
of $A$-modules is commutative, where
the vertical isomorphisms are given by the correspondence $x\mapsto t^i\otimes x$.
Identifying $\R'(\M)$ with $\bigoplus_{i\in\Z}M_i$ and $\R'(\calN)$ with $\bigoplus_{i\in\Z}N_i$, we can view $\varphi |_i=\varphi_i$.

If $f: M\to N$ is an homomorphism of $A$-modules satisfying $f(M_i)\subseteq N_i$ for each $i\in\Z$,
then we have the homomorphism $\varphi: \R'(\calM)\to\R'(\calN)$ of graded $\R'(\F)$-modules defined by $$\varphi\bigg(\sum_{i\in\Z}t^i\otimes x_i\bigg)= \sum_{i\in\Z}t^i\otimes f(x_i),$$
which is referred to as {\it the graded $\R'(\F)$-homomorphism from
extended from $f$} in this paper.
Let us note that the converse of this fact holds true as follows.

\begin{lem}\label{restriction1}
Let $\varphi :\R'(\calM)\to\R'(\calN)$ be a homomorphism of graded $\R'(\F)$-modules. Then, taking an integer $n$ such that $M_n=M$ and $N_n=N$,
one has the diagram
\[
\xymatrix
{
M \ar[r]^{\varphi |_n} \ar@{}[d]|{\bigcup}^|& N \ar@{}[d]|{\bigcup}^|\\
M_k \ar[r]^{\varphi |_k}   & N_k
}
\] 
is commutative for every integer $k$, which illustrates that the map $\varphi$ is the graded $\R'(\F)$-homomorphism extended from $\varphi |_n$.
\end{lem}

\begin{proof}
Let $i$ and $j$ be integers such that $i\le j$.
It is enough to show that  
the diagram
\[
\xymatrix
{
M_i \ar[r]^{\varphi |_i} \ar@{}[d]|{\bigcup}^|& N_i \ar@{}[d]|{\bigcup}^|\\
M_j \ar[r]^{\varphi |_j}   & N_j
}
\] 
is commutative. 
Take any element $x\in M_j$, and then $x\in M_i$, as $M_j\subseteq M_i$. 
We note that $t^{i-j}\in \R'(\F)$. Hence, 
\begin{align*}
t^i\otimes \varphi |_j(x)
=t^{i-j}\cdot\big(t^j\otimes \varphi|_j(x)\big)
=t^{i-j}\cdot \varphi (t^j\otimes x)
&=\varphi\big(t^{i-j}\cdot(t^j\otimes x)\big)\\
&=\varphi\big(t^i\otimes x\big)\\
&=t^i\otimes \varphi|_i(x),
\end{align*}
whence $t^i\otimes \varphi |_j(x)=t^i\otimes \varphi|_i(x)$. Thus, $\varphi |_j(x)=\varphi|_i(x)$, as desired. 
For the last assertion, we have $\varphi |_n(M_k)=\varphi |_k(M_k)\subseteq N_k$ for every integer $k$, and therefore $\varphi$ is the graded $\R'(\F)$-homomorphism extended from $\varphi |_n$.
\end{proof}

For each homomorphism $\varphi :\R'(\calM)\to\R'(\calN)$ of graded $\R'(\F)$-modules, we put $\overline{\varphi}=\varphi\otimes_{\R'(\F)}\G(\F)$,
namely $\overline{\varphi}: \G(\calM)\to\G(\calN)$, 
which is a homomorphism of graded $\G(\F)$-modules.
Furthermore, for each $i\in\Z$, we denote by $$\overline{\varphi}|_i: M_i/M_{i+1}\to N_i/N_{i+1}$$ the $A$-homomorphism given by $\overline{\varphi}|_i(\overline{x})=\overline{\varphi|_i(x)}$ for each $x\in M_i$, where 
$\overline{x}$ denotes the image of $x$ in $M_i/M_{i+1}$ and $\overline{\varphi|_i(x)}$ is the image of $\varphi|_i(x)$ in $N_i/N_{i+1}$.
That is to say, identifying $\G(\M)=\bigoplus_{i\in\Z} M_i/M_{i+1}$ and $\G(\calN)=\bigoplus_{i\in\Z} N_i/N_{i+1}$, we can view $\overline{\varphi}|_i$ as the $i^{\th}$ homogeneous component of $\overline{\varphi}$. 

For each Noetherian $\Z$-graded ring $S$, each graded $S$-homomorphism $\phi: E\to E'$, and each integer $n$, we denote by $$\phi_{\ge n}: E_{\ge n}\to E'_{\ge n}$$ the graded $S_{\ge 0}$-homomorphism given by $\phi_{\ge n}(x)=\phi (x)$ for every $x\in E_{\ge n}$.

We say that an $\F$-filtration $\calL=\{L_i\}_{i\in\Z}$ of an $A$-submodule $L$ of $M$ is {\it an $\F$-subfiltration of $\M$} if $L_i\subseteq M_i$ for each $i\in\Z$, and then we represent it as $\calL\subseteq \M$. Let us consider a question of when the equality $\calL= \M$ holds true. Our answer is the following, which is an essentially given by \cite[Theorem 3.2]{GI}.

\begin{lem}\label{embed}
Let $\calL=\{L_i\}_{i\in\Z}$ be an $\F$-subfiltration of $\M$ and let $\varphi: \R'(\calL)\hookrightarrow \R'(\M)$ denote the inclusion map. 
Let $n$ be an integer.
Assume that the homomorphism $$\overline{\varphi}_{\ge n}: \G(\calL)_{\ge n}\to\G(\M)_{\ge n}$$ of graded $\G(\F)$-modules is injective and $L_n\supseteq M_{n+1}$. Then the equality $L_i=M_i$ holds for each integer $i\ge n+1$.
\end{lem}

\begin{proof}
For every integer $i$, we have $\ker \overline{\varphi}|_i=L_i\cap M_{i+1}/L_{i+1}$. Hence, $$L_i\cap M_{i+1}=L_{i+1}$$ for each integer $i\ge n$ because  
$\overline{\varphi}_{\ge n}$ is injective.
Since $L_n\supseteq M_{n+1}$, it follows from the equality above that $L_{n+1}=M_{n+1}$, and then we have $L_{n+1}\supseteq M_{n+2}$. Thus, we see the required equality $M_{i}=L_{i}$ holds true by using induction on $i\ (\ge n+1)$.
\end{proof}

For each $m\in\Z$, we define $\M(m)=\{M_{i+m}\}_{i\in\Z}$, which forms an $\F$-filtration of $M$, and then $\R'(\M(m))=\sum_{i\in\Z} At^{i}\otimes_A M_{i+m}$. Hence, 
\begin{center}
$\R'(\M(m))\cong\R'(\M)(m)$\quad and\quad $\G(\M(m))\cong\G(\M)(m)$
\end{center}
as a graded $\R'(\F)$-module. Notice that 
$$\R'(\M(m))_{\ge -m}\cong \R(\M)(m)$$ 
as a graded $\R(\F)$-module. Let us now state the following.

\begin{lem}\label{y} 
Let $n$ be an integer.
Then the following two conditions are equivalent.
\begin{enumerate}
\item $\R' (\M)_{\ge n}\cong \R(\F)(-n)$ as a graded $\R(\F)$-module. \vspace{1mm}
\item There is $y\in M$ such that $0:_Ay=0$ and $M_{i}=F_{i-n} y$ for every integer $i\ge n$.
\end{enumerate}
\end{lem}

\begin{proof}
$(1)\Rightarrow (2)$ 
Let $\phi: \R(\F)(-n)
\xrightarrow{\sim}\R' (\M)_{\ge n}$ stand for the given isomorphism of graded $\R(\F)$-modules. We write $\phi (1)=t^n\otimes y$ for some $y\in M_n$. Since $\phi$ is surjective, $$\R' (\M)_{\ge n}=\R(\F)\cdot (t^n\otimes y)=\sum_{i\ge 0}F_it^{i}\cdot (t^n\otimes y)=\sum_{i\ge n}At^{i}\otimes F_{i-n}y.$$ Since $\R' (\M)_{\ge n}=\sum_{i\ge n}At^{i}\otimes_A M_{i}$, we get $M_i=F_{i-n}y$ for all integers $i\ge n$, as desired. 
Furthermore, we have $0:_Ay=0$ because the $n^{\th}$ homogeneous component of $\phi$ induces that $A\cong M_n=Ay$.

$(2)\Rightarrow (1)$ 
We set $f: A\to M$ to be the $A$-homomorphism given by $f(1)=y$ and $\varphi : \R'(\F(-n))\to\R'(\calM)$ to be the graded $\R'(\F)$-homomorphism extended from $f$. Let $i$ be an integer satisfying $i\ge n$. Then the homomorphism $\varphi |_{i}: F_{i-n}\to M_i$ is a restriction of $f$ by Lemma \ref{restriction1}.  It is routine to check that $\varphi |_{i}$ is bijective because
 $M_{i}=F_{i-n} y$ and $0:_Ay=0$, 
whence so is the homomorphism $\varphi_{\ge n}$. Thus, $\R(\F)(-n)\cong\R' (\M)_{\ge n}$ as a graded $\R(\F)$-module.
\end{proof}

An $\F$-filtration $\L=\{L_i\}_{i\in\Z}$ of an $A$-module $L$ is said to be {\it separated} (resp. {\it strongly separated}), if $\bigcap_{i\in \Z}L_i=0$ (resp. $\bigcap_{i\in\Z}(N+L_i)=N$ for each $A$-submodule $N$ of $L$).
The following proposition, which is a generalization of results \cite[Lemma 4.4]{GN} and \cite[Lemma 3.2]{TVZ},
holds a pivotal role in this paper. 

\begin{prop}\label{embed2}
Assume that $\F$ is  separated and $\M$ is strongly separated.
Let $n$ be an integer. 
Then the following three conditions are equivalent.\vspace{1mm}
\begin{enumerate}
\item $\R' (\M)_{\ge n}\cong \R(\F)(-n)$ as a graded $\R(\F)$-module. \vspace{2mm}
\item $\G (\M)_{\ge n}\cong \G(\F)(-n)$ as a graded $\G(\F)$-module. \vspace{2mm}
\item
\begin{enumerate}
\item[{\rm (i)}] $M_{n}\cong A$ as an $A$-module and 
\item[{\rm (ii)}]  there exists a monomorphism $\G(\F)(-n)\hookrightarrow \G (\M)$ of graded $\G(\F)$-modules. 
\end{enumerate}
\end{enumerate}
\end{prop}

\begin{proof}
$(1)\Rightarrow (2)$ By Lemma \ref{y}, we can choose $y\in M$ such that $M_{i}=F_{i-n} y$ for all integer $i\ge n$ and $0:_Ay=0$. Put $L_i=F_{i-n}y$ for every $i\in\Z$ and $\calL=\{L_i\}_{i\in\Z}$, which is an $\F$-filtration of $Ay$, and then $L_i=M_i$ for every integer $i\ge n$. 
Let $f: A\to M_n=Ay$ stand for the $A$-isomorphism given by $f(1)=y$.
Set $$\varphi : \R'(\F(-n))\xrightarrow{}\R'(\calL)$$ to be the graded $\R'(\F)$-homomorphism extended from $f$. Then we observe that $\varphi$ is bijective.
Consider $\varphi\otimes_{\R'(\F)}\G(\F)$, and we have
the isomorphism $$\overline{\varphi}: \G(\F(-n))\xrightarrow{\sim} \G (\calL)$$
of $G(\F)$-modules. Since $\overline{\varphi}_{\ge n}$ is also bijective, we obtain from $\G (\calL)_{\ge n}\cong\G (\M)_{\ge n}$ as a graded $\G(\F)$-module that $\G(\F)(-n)\cong \G (\M)_{\ge n}$ as a graded $\G(\F)$-module, as required.

$(2)\Rightarrow (3)$ Let $\psi: \G(\F)(-n) \xrightarrow{\sim}\G (\M)_{\ge n}$ stand for the given isomorphism of graded $\G(\F)$-modules. It is enough to show that the condition (i) holds true. We write $\psi (1)=\overline{t^n\otimes y}$ for some $y\in M_n$, where $\overline{t^n\otimes y}$ is the image of $t^n\otimes y$ in $\R' (\M)/t^{-1}\R' (\M)$. 
Set $$\varphi : \R'(\F(-n))\to\R'(\calM)$$ to be the graded $\R'(\F)$-homomorphism extended  from $f:A\to M$ given by $f(1)=y$.
Look at the homomorphism $$\overline{\varphi}: \G(\F(-n))\to \G (\M)$$ of $G(\F)$-modules. Then we can view $\overline{\varphi}_{\ge n}=\psi$, since 
$\G(\F(-n))_{\ge n}\cong \G(\F)(-n)$ as a graded $\G(\F)$-module.

For every integer $i\ge n$, we have $\overline{\varphi}|_i: F_{i-n}/F_{i-n+1}\to M_i/M_{i+1}$ is surjective because so is $\psi\, (=\overline{\varphi}_{\ge n})$, and hence the equality $$M_i=F_{i-n}y+M_{i+1}$$ holds true.
Let $k$ be an integer such that $k\ge n$. By using induction on $i\, (\ge k)$, we obtain that $M_{k}=F_{k-n}y+M_{i+1}$,
which yields the equality $$M_k=\bigcap_{i\ge k}(F_{k-n}y+M_{i+1}),$$ and thus $M_k=F_{k-n}y$ because $\M$ is strongly separated. In particular, $M_n=Ay$.

We must show $0:_Ay=0$.
For every integer $i\ge n$, we have $$\overline{\varphi}|_i: F_{i-n}/F_{i-n+1}\to M_i/M_{i+1}=F_{i-n}y/F_{i-n+1}y$$ is injective because so is $\psi$, and hence the equality
$$F_{i-n+1}=F_{i-n+1}y:_{F_{i-n}}y$$  holds true because $\ker \overline{\varphi}\big|_i=(F_{i-n+1}y:_{F_{i-n}}y)/F_{i-n+1}$. That is to say, $F_{i}=F_{i}y:_{F_{i-1}}y$ for all integers $i\ge 1$.
By using induction on $i\ge 1$, we obtain that $F_{i}=F_{i}y:_{A}y$.
Hence, 
\begin{center}
$\bigcap_{i\ge 1}F_{i}=\bigcap_{i\ge 1}\big(F_{i}y:_{A}y\big)=\big(\bigcap_{i\ge 1}F_{i}y\big):_{A}y.$
\end{center}
Since $\F$ and $\M$ are separated, we have $\bigcap_{i\ge 1}F_{i}=0$ and $\bigcap_{i\ge 1}F_{i}y\subseteq \bigcap_{i\ge 1}M_{i+n}=0$, so that $0:_Ay=0$. Thus, $M_{n}\cong A$ as an $A$-module.

$(3)\Rightarrow (1)$ Since $M_{n}\cong A$ as an $A$-module, we can find an element $y\in M_n$ such that $M_{n}=Ay$ and $0:_Ay=0$.
Put $L_i=F_{i-n}y$ for every $i\in\Z$ and set $\calL=\{L_i\}_{i\in\Z}$, which is an $\F$-filtration of $Ay$, and then $\calL\subseteq \M$ and $L_n=M_n=Ay$. 

Let $\iota: Ay\hookrightarrow M$ be the inclusion map and let $f:A\xrightarrow{\sim} Ay$ be the $A$-isomorphism such that $f(1)=y$. Set $\varphi_1: \R' (\L)\hookrightarrow\R' (\M)$ and $\varphi_2: \R' (\F(-n))\xrightarrow{\sim}\R' (\L)$ to be the graded $\R' (\F)$-homomorphisms extended from $\iota$ and $f$, respectively. 
Look at the induced
homomorphisms
\begin{center}
$\overline{\varphi_1 }: \G (\L)\to\G (\M)$\quad and \quad $\overline{\varphi_2 }: \G(\F(-n))\to \G (\L)$
\end{center}
of graded $\G(\F)$-modules.
We notice that $\overline{\varphi_2}$ is bijective because so is $\varphi_2$. Hence, $\G (\L)$ is a free graded $\G (\F)$-module of rank one, and hence we represent 
$$\G (\L)=\G (\F)\cdot \overline{t^n\otimes y},$$
because $\varphi_2(t^n\otimes 1)=t^n\otimes y$, where $\overline{t^n\otimes y}$ denotes the image of $t^n\otimes y$ in $\R'(\L)/t^{-1}\R'(\L)$.

By the composition of  the natural isomorphism $\G(\F(-n))\cong\G(\F)(-n)$ and the given monomorphism $\G(\F)(-n)\hookrightarrow \G (\M)$, we get the monomorphism
$\psi_1: \G(\F(-n))\hookrightarrow \G (\M)$ of graded $\G(\F)$-modules. 
Moreover, we put $\psi_2=\psi_1\circ\overline{\varphi_2}^{-1}$, namely $$\psi_2: \G(\L)\hookrightarrow \G (\M),$$which is a monomorphism of graded $\G(\F)$-modules. Since $L_n=M_n=Ay$, there exists $a\in A$ such that $\psi_2(\overline{t^n\otimes y})=\overline{t^n\otimes ay}$. 
We denote by $$\hat {a}: \G (\M)\to \G (\M)$$ the multiplication map by $a$. Since $\G(\L)$ is a free $\G(\F)$-module of rank one generated by the element $\overline{t^n\otimes y}$ as mentioned above, the homomorphism $\psi_2$ is determined by the correspondence $\overline{t^n\otimes y}\mapsto\overline{t^n\otimes ay}$, so that it is routine to check that the equality $$\psi_2=\hat {a}\circ\overline{\varphi_1},$$ holds true. Thus, we obtain that $\overline{\varphi_1}$ is injective, since so is $\psi_2$. 

The equality $L_n=M_n$ implies that $L_n\supseteq M_{n+1}$, and therefore, we get $L_i=M_i$ for all integers $i\ge n+1$ by Lemma \ref{embed}. This means that $M_{i}=F_{i-n} y$ for every integer $i\ge n$, and thus $\R(\F)(-n)\cong\R(\M)_{\ge n}$ as a graded $\R(\F)$-module by Lemma \ref{y} (recall that $0:_Ay=0$). This completes the proof.  
\end{proof}

We notice that, when 
$\F'=\{F'_i\}_{i\in\Z}$ is a filtration of ideals in $A$ such that $\F\subseteq\F'$, the ring $\R'(\F)$ becomes a subring of $\R'(\F')$. We conclude this section with the following lemma that will be instrumental in Section 4. Let us prove it.

\begin{lem}\label{extendF}
Let $\F' = \{F'_i\}_{i\in\Z}$ be a filtration of ideals in $A$ satisfying $\F \subseteq \F'$. Suppose $\calK$ is a graded $\R'(\F')$-module with a graded $\R'(\F)$-isomorphism $\phi: \R'(\M) \xrightarrow{\sim} \calK$. Then $\M$ naturally extends to an $\F'$-filtration of $M$, and $f$ becomes a graded isomorphism of graded $\R'(\F')$-modules.
\end{lem}

\begin{proof}
Take any element $\alpha\in\R(\F')$ and any element $x\in\R(\M)$. We define the operation $\alpha\ast x$ as $\phi^{-1}\big(\alpha\, \phi(x)\big)$. Then it is routine to check that $\R(\M)$ extends to an graded $\R(\F')$-module such that $\phi$ becomes a graded $\R'(\F')$-isomorphism. Let $\cdot$ stand for the ordinary operation of $A[t,t^{-1}]$ to $A[t,t^{-1}]\otimes_AM$. We shall prove that $\alpha\ast x=\alpha\cdot x$. 

Take an integer $k\le 0$ such that $t^k\cdot\alpha \in \R'(\F)$. Multiplying the element $t^k$ to the both sides of the equality $\alpha\ast x=\phi^{-1}\big(\alpha\, \phi(x)\big)$, we get $t^k\cdot(\alpha\ast x)=t^k\cdot\phi^{-1}\big(\alpha\, \phi(x)\big)$.
Since $t^k\in\R'(\F)$, the right hand side can be calculated as follows:
\begin{align*}
t^k\cdot\phi^{-1}\big(\alpha\, \phi(x)\big)=\phi^{-1}\big(t^k\, (\alpha\, \phi(x)) \big)=\phi^{-1}\big((t^k\cdot\alpha)\, \phi(x) \big)
&=
(t^k\cdot \alpha)\cdot \phi^{-1}\big( \phi (x) \big)\\
&=(t^k\cdot \alpha)\cdot x\\
&=t^k\cdot (\alpha\cdot x),
\end{align*}
and hence
$t^k\cdot(\alpha\ast x)=t^k\cdot (\alpha\cdot x)$, so that
$\alpha\ast x=\alpha\cdot x$ because the element $t^k$ is a nonzero divisor on $A[t,t^{-1}]\otimes_AM$. This implies that $\R'(\M)$ is a graded $\R(\F')$-submodule of $A[t,t^{-1}]\otimes_AM$.
Thus, $F_i'\cdot M_j\subseteq M_{i+j}$, as required. 
\end{proof}

\section{Extended canonical $\F$-filtrations}
In this section, we will examine an $\F$-filtration describing the graded canonical module of an extended Rees module, which was originally introduced by Herzog, Simis, and Vasconcelos in \cite{HSV}. 
Notably, Ulrich further investigated this topic in \cite[Section 2]{U}, determining the graded canonical module of the extended Rees algebra and emphasizing its significance.

First of all, let us recall the notion of graded canonical modules. 
For each local ring $(B,\n)$, let $\widehat{\,B\, }$ denote the $\n$-adic completion of $B$ and let $\E_B$ stand for the injective envelope of $B/\n$ with respect to $B$.  
Let $R=\bigoplus_{i\in\Z}R_i$ be a $\Z$-graded Noetherian ring and $L$ a finitely generated graded $R$-module. 
For each homogeneous multiplicatively closed subset $T$ of $R$, we denote by $T^{-1}L$ the localization of $L$ with respect to $T$, and then $T^{-1}L$ becomes a $\Z$-graded $R$-module whose $i^{\th}$ homogeneous component given by $[T^{-1}L]_i=\left\{x/y\in T^{-1}L\mid x\in L_{i+\deg y},\ y\in T\right\}$ for every $i\in\Z$.

Let $\p$ be a homogeneous prime ideal of $R$.
When $T$ is the set of homogeneous element of $R$ not belonging to $\p$, we set $L_{(\p)}=T^{-1}L$, which is called a homogeneous localization of $L$ at $\p$ (in the sense of \cite[Section 1.5]{BH}). 
Hence, $R_{(\p)}$ is a $\Z$-graded Noetherian ring  with a unique homogeneous maximal ideal $\p R_{(\p)}$ and $L_{(\p)}$ is a finitely generated graded $R_{(\p)}$-module. Note that $[R_{(\p)}]_0$ is a Noetherian local ring.

We define $\K_{L_{(\p)}}$ as a graded $R_{(\p)}$-module such that   
$$
\K_{L_{(\p)}}\otimes_{[R_{(\p)}]_0} \widehat{\, [R_{(\p)}]_0}\cong\bigoplus_{i\in\Z} \Hom_{[R_{(\p)}]_0}
\big([\H_{\p R_{(\p)}} ^{\dim L_{\p}}(L_{(\p)})]_{-i}, \E_{[R_{(\p)}]_0}\big)
$$ 
as a graded $R_{(\p)}$-module, which is called a graded canonical module of $L_{(\p)}$. When $\p$ is a unique homogeneous maximal ideal of $R$, we simply represent $\K_{L_{(\p)}}$ by $\K_{L}$, as $L_{(\p)}=L$.

Let $\varphi:S\to R$ be a finite homomorphism of $\Z$-graded Noetherian rings such that $S$ is a Gorenstein ring (i.e., $S_P$ is a Gorenstein local ring for all $P\in\Spec S$). Since $S_{(\varphi^{-1}(\p))}$ is a Gorenstein ring, there exists an integer $a$ such that $\K_{S_{(\varphi^{-1}(\p))}}\cong S_{(\varphi^{-1}(\p))}(a)$ as an graded $S_{(\varphi^{-1}(\p))}$-module (see \cite[Proposition 3.6.11]{BH}).
Put $m=\dim S_{\varphi^{-1}(\p)}-\dim L_{\p}$.
Then, by the same argument as in the proof of \cite[Corollary (36.14)]{HIO}, we obtain that 
$$
\K_{L_{(\p)}}\cong \Ext^{m}_{S_{(\varphi^{-1}(\p))}}(L_{(\p)}, S_{(\varphi^{-1}(\p))}(a))
$$ 
as a graded $R_{(\p)}$-module.
In addition, if $\varphi$ is surjective, then $L_{(\p)}=L_{(\varphi^{-1}(\p))}$, so that 
$$
\K_{L_{(\p)}}\cong \Ext^{m}_{S}(L, S(a))_{(\varphi^{-1}(\p))}=\Ext^{m}_{S}(L, S(a))_{(\p)}
$$
as a graded $R_{(\p)}$-module, and moreover, if  $\p$ is a unique homogeneous maximal ideal of $R$, then we can represent   
$$
\K_{L}\cong \Ext^{m}_{S}(L, S(a))
$$
as a graded $R$-module. 
We make use of the notation introduced above for the next

\begin{note*}\label{T}
Assume that the homomorphism $\varphi:S\to R$ is surjective. Let $T$ be a homogeneous multiplicatively closed subset of the graded ring $S$. Suppose that both rings $R$ and $T^{-1}R$ have a unique homogeneous maximal ideal. Then $T^{-1}\K_L\cong \K_{T^{-1}L}(c)$ as a graded $T^{-1}S$-modules for some integer $c$, provided $T^{-1}\K_L\neq 0$. 
\end{note*}

\begin{proof}
Look at the epimorphism
$$T^{-1}\varphi :T^{-1}S\to T^{-1}R$$ 
of graded rings given by the localization of $\varphi$ with respect to $T$. Let $\fkM$ be a unique homogeneous maximal ideal of $T^{-1}R$.  We put $\fkN=(T^{-1}\varphi)^{-1}(\fkM)$, which is a homogeneous prime ideal of $T^{-1}S$.
Recall that $\K_{L}\cong \Ext_{S}^m(L, S(a))$ as a graded $R$-module, and we observe that
$$
T^{-1}\K_{L}\cong T^{-1}\Ext_{S}^m(L, S(a))\cong \Ext_{T^{-1}S}^m(T^{-1}L, T^{-1}S(a))
$$
as a graded $T^{-1}R$-module. Hence, $m=\dim [T^{-1}S]_{\fkN}-\dim [T^{-1}L]_{\fkM}$ because $T^{-1}\K_{L}\neq 0$.
Since the ring $T^{-1}S$ is Gorenstein, so is the graded ring $[T^{-1}S]_{\left(\fkN\right)}$, and hence we can find an integer $b$ such that $\K_{[T^{-1}S]_{\left(\fkN\right)}}\cong[T^{-1}S]_{\left(\fkN\right)}(b)$ as a graded $[T^{-1}S]_{\left(\fkN\right)}$-module.
Then we obtain that $$\K_{T^{-1}L}\cong\Ext_{T^{-1}S}^m(T^{-1}L, T^{-1}S(b))\cong \Ext_{T^{-1}S}^m(T^{-1}L, T^{-1}S(a))(b-a)$$ as a graded $T^{-1}R$-module.
Thus, $\K_{T^{-1}L}\cong T^{-1}\K_{L}\, (b-a)$
as a graded $T^{-1}R$-module. 
\end{proof}

In the rest of this section, let $(A, \m)$ be a Noetherian local ring, $\F=\{F_i\}_{i\in\mathbb{Z}}$ a filtration of ideals in $A$, $M$ a nonzero finitely generated $A$-module.
Assume that $\R(\F)$ is Noetherian.
Then $\R'(\F)$ and $\G(\F)$ are also Noetherian. We put $\fkM=\m\R(\F)+\R(\F)_{+}$ and $\fkM'=\fkM\R'(\F)+t^{-1}\R'(\F)$, which are a unique homogeneous maximal ideal of $\R(\F)$ and $\R'(\F)$, respectively. In this section, we put 
$$\M=\{F_iM\}_{i\in\Z},$$ which is an $\F$-filtration of $M$.
Then $\R'(\M)$ and $\G(\M)$ are finitely generated $\R'(\F)$-modules with $\dim \R'(\M)=\dim M+1$ and $\dim \G(\M)=\dim M$ by
\cite[Theorem 4.5.6]{BH}. 
We notice that $\M=\F$ when $M=A$.
Throughout this section, we assume that the base ring $A$ is a homomorphic image of a Gorenstein local ring $C$.
We denote by $\pi : C\to A$ a given epimorphism of rings. 
We set $\K_M=\Ext_{C}^{\dim C-\dim M}(M, C)$, which is called the canonical module of $M$.
To begin with, let us state the next result, which is given by \cite[Theorem 2.1]{GI} in the case where the base ring $A$ is Cohen-Macaulay and $M=A$.

\begin{lem}[cf. \cite{GI}]\label{cf}
There exists an $\F$-filtration $\w=\{\w_n\}_{n\in\Z}$ of $\K_M$ such that $$\K_{\R'(\M)}\cong\R'(\w)$$ as a graded $\R'(\F)$-module. Hence, one has a monomorphism $\G(\w)(-1)\hookrightarrow \K_{\G(\M)}$ of graded $\G(\F)$-modules. 
\end{lem}

To show the lemma above, let us fix the following notation.
Since the ring $\R'(\F)$ is Noetherian, we can represent
$\R'(\F)=A[a_1t^{n_1}, a_2t^{n_2}, \dots, a_\ell t^{n_\ell}, t^{-1}]$
for some nonnegative integer $\ell$, some elements $a_1,a_2, \dots, a_\ell\in A$, and some positive integers $n_1,n_2, \dots, n_\ell$.
Let $X_1,X_2,\dots, X_\ell, Y$ be indeterminates over $C$ and put $\P'=C[X_1,X_2,\dots, X_\ell, Y]$ that is a $\Z$-graded polynomial ring over $C$ such that $\deg X_i=n_i$ for all $1\le i\le \ell$, $\deg Y=-1$, and $\deg c=0$ for all $c\in C\setminus \{ 0 \}$. 
We set
\begin{center}
$\varphi :\P'\to \R'(\F)$ 
\end{center}
to be the epimorphism of graded $C$-algebras such that $\varphi (X_i)=a_it^{n_i}$ for all $1\le i\le \ell$, $\varphi (Y)=t^{-1}$, and $\varphi (c)=\pi (c)$ for all $c\in C$.  
Since the graded ring $\P'\, (=\P'_{(\varphi^{-1}(\fkM'))})$ is Gorenstein, there exists an integer $a$ such that $\K_{\P'}\cong\P'(a)$ as a graded $\P'$-module.
Put $m=\dim \P' -\dim \R'(\M)$. 
Then we have
\begin{center}
$\K_{\R'(\M)}\cong \Ext_{\P'}^m(\R'(\M), \P'(a))$\ \ and\ \ $\K_{\G(\M)}\cong \Ext_{\P'}^{m+1}(\G(\M), \P'(a))$
\end{center} 
as a graded $\P'$-module. Taking the $\P'(a)$-dual of the short exact sequence $$0\to \R'(\M)(1)\xrightarrow{t^{-1}} \R'(\M)\to \G(\M) \to 0$$ of graded $\P'$-modules induced by multiplication by $t^{-1}$ on $\R'(\M)$, we get the short exact sequence $0\to \K_{\R'(\M)}\xrightarrow{t^{-1}} \K_{\R'(\M)}(-1)\to \K_{\G(\M)}$ of graded $\R'(\F)$-modules. Hence, the element $t^{-1}$ is a nonzero divisor on $\K_{\R'(\M)}$, and furthermore, we obtain the monomorphism $$\big(\K_{\R'(\M)}/t^{-1} \K_{\R'(\M)}\big)(-1)\hookrightarrow \K_{\G(\M)}$$ of graded $\G(\F)$-modules. 
We now ready to prove Lemma \ref{cf}.

\begin{proof}[Proof of Lemma \ref{cf}]
Set $T=\{Y^n\mid n\in \Z,\ n\ge 0\}$, which is a multiplicatively closed subset of $\P'$.
Since $\varphi (Y)=t^{-1}$, we have $T^{-1}\R'(\F)=
A[t,t^{-1}]\cdot\R'(\F)
=A[t, t^{-1}]$, which has a unique homogeneous maximal ideal $\m A[t, t^{-1}]$. We note that the equality $m=\dim T^{-1}\P'-\dim T^{-1}\R'(\M)$ holds true.  
Therefore, it follows that $$T^{-1}\K_{\R'(\M)}\cong T^{-1}\Ext_{\P'}^m(\R'(\M), \P'(a))\cong \Ext_{T^{-1}\P'}^m(T^{-1}\R'(\M), T^{-1}\P'(a))\neq 0.$$ From the note as above, we obtain that  
$T^{-1}\K_{\R'(\M)}\cong\K_{T^{-1}\R'(\M)}(c)=\K_{A[t,t^{-1}]\otimes_A M}(c)$
as a graded $T^{-1}\R(\F)$-module for some integer $c$. 
Moreover, it follows from \cite[Proposition 4.1]{A} that $$\K_{A[t,t^{-1}]\otimes_A M}(c)\cong (A[t,t^{-1}]\otimes_A\K_M)(c)\cong A[t,t^{-1}]\otimes_A\K_M$$ as a graded $T^{-1}\R(\F)$-module (where the latter equivalence follows from the multiplication map by $t^{-c}$).
Thus, we observe that $$T^{-1}\K_{\R'(\M)}\cong A[t,t^{-1}]\otimes_A\K_M$$ as a graded $T^{-1}\R(\F)$-module. Since $t^{-1}$ is a nonzero divisor on $\K_{\R'(\M)}$,
there exists an $\F$-filtration
$\w=\{\w_i\}_{i\in\Z}$ of $\K_M$ such that $\K_{\R'(\M)}\cong\R'(\w)$ as a graded $\R'(\F)$-module by Lemma \ref{filtration}.
Since $\G(\w)=\R'(\w)/t^{-1}\R'(\w)$, the last assertion directly follows from the monomorphism $\big(\K_{\R'(\M)}/t^{-1} \K_{\R'(\M)}\big)(-1)\hookrightarrow \K_{\G(\M)}$ of graded $\G(\F)$-modules.
\end{proof}

Let us refer to an $\F$-filtration $\w=\{\w_n\}_{n\in\Z}$ of $\K_M$ with a graded $\R'(\F)$-isomorphism $\K_{\R'(\M)}\xrightarrow{\sim}\R'(\w)$ as {\it an extended canonical $\F$-filtration of $\K_M$ with respect to $\M$} in this paper.

In what follows, let $\w=\{\w_i\}_{i\in\Z}$ denote an extended canonical $\F$-filtration of $\K_M$ with respect to $\M$ and assume that $F_1\neq A$. We define
$$
\ai '(\M)=-\inf\{i\in\Z\mid \w_i\neq\K_M\},
$$
which will be used to prove Theorem \ref{main2}.
We note that $\ai '(\M)\in\Z$. In fact, when $\w_i=\K_M$ for all $i\in\Z$, we have $\K_{\R'(\M)}\cong \R(\w)=A[t, t^{-1}]\otimes_{A}\K_{M}$ as a graded $\R'(\F)$-module, and then we see $A[t, t^{-1}]\otimes_{A}\K_{M}$ is a finitely generated $\R'(\F)$-module because so is $\K_{\R'(\M)}$. However, this is impossible because $F_1\neq A$, and thus $\ai '(\M)\in\Z$.

Notice that $\w_{-\ai '(\M)-1}=\K_M\supsetneq \w_{-\ai '(\M)}$, and we see that the equality 
$$
-\ai '(\M)-1=\inf\{i\in\Z\mid [\K_{\R'(\M)}/t^{-1} \K_{\R'(\M)}]_i\neq 0\}
$$
holds true,
because
$$\K_{\R'(\M)}/t^{-1} \K_{\R'(\M)}\cong\R'(\w)/t^{-1}\R'(\w)=\G(\w)
\cong\bigoplus_{i\in\Z}\w_{i}/\w_{i+1}
=\bigoplus_{i\ge -\ai '(\M)-1}\w_{i}/\w_{i+1}$$ as a graded $\G(\F)$-module (recall that $\w_{-\ai '(\M)-1}/ \w_{-\ai '(\M)}\neq 0$). 
Hence, the number $\ai '(\M)$ does not depend on the choice of an extended canonical $\F$-filtration of $\K_M$ with respect to $\M$.
When $\left(\K_{\R'(\F)}\big/t^{-1}\K_{\R'(\F)}\right)(-1)\cong \K_{\G(\M)}$ as a graded $\G(\F)$-module, we obtain the equality $\ai '(\M)=\ai (\G(\M))$. Although it is uncertain whether this equality always holds, we note the following. 

\begin{lem}\label{-a-1}
The inequality
$\ai '(\M)\le \ai (\G(\M))$ holds true,
and hence $\w_j=\K_M$
for all integers $j<-\ai (\G(\M))$.
\end{lem}

\begin{proof}
We observe the equality $
\ai '(\M)=-\min\{i\in\Z\mid [\K_{\R'(\M)}/t^{-1} \K_{\R'(\M)}(-1)]_i\neq 0\}
$. Thus, the assertion arises from the monomorphism $\big(\K_{\R'(\M)}/t^{-1} \K_{\R'(\M)}\big)(-1)\hookrightarrow \K_{\G(\M)}$ of graded $\G(\F)$-modules.
\end{proof}

For each $\F$-filtration $\calN=\{N_i\}_{i\in\Z}$ of an $A$-module $N$ and 
each positive integer $n$, we define $\calN^{(n)}=\{N_{ni}\}_{i\in\Z}$, which is an $\F^{(n)}$-filtration of $N$, and then we can write $\R'(\calN^{(n)})
=\sum_{i\in\Z}At^i\otimes_A N_{ni}$. Hence, it is routine to check that 
$\R'(\F^{(n)})
\cong\R'(\F)^{(n)}$ as rings and that
$\R'(\calN^{(n)})
\cong\R'(\calN)^{(n)}$ as a graded $\R(\F^{(n)})$-module, where  $\R'(\F)^{(n)}$ (resp. $\R'(\calN)^{(n)}$) is the $n^{\th}$ Veronesian subring ring of $\R'(\F)$ (resp. submodule of $\R'(\calN)$).

We shall recall briefly the concept of Veronesian subrings and submodules. Let $R=\bigoplus_{i\in\Z}R_i$ be a $\Z$-graded Noetherian ring.
For each positive integer $n$ and each graded $R$-module $E$, we set  $E^{(n)}=\bigoplus_{i\in\Z}E_{ni}$. Then $R^{(n)}$ is a $\Z$-graded ring whose grading is given by $[R^{(n)}]_i=R_{ni}$ for each integer $i$. 
Furthermore, $E^{(n)}$ is a graded $R^{(n)}$-module whose $i^{\th}$ homogeneous component is given by $[E^{(n)}]_i=E_{ni}$. The ring $R^{(n)}$ (resp. the module $E^{(n)}$) is called the $n^{\th}$ Veronesian subring of $R$ (resp. submodule of $E$). We refer the reader to \cite[\S 47]{HIO} for details of Veronesian subrings and submodules. 
In particular, it can be observed that $$\K_{\R'(\M^{(n)})}\cong\K_{\R'(\M)^{(n)}}\cong(\K_{\R'(\M)})^{(n)}\cong \R'(\w)^{(n)}\cong \R'(\w^{(n)})=\sum_{i\in\Z}At^i\otimes\w_{ni}$$ as a graded $\R(\F^{(n)})$-module. Hence, $\w^{(n)}=\{\w_{ni}\}_{i\in\Z}$ is an extended canonical $\F^{(n)}$-filtration of $\K_M$ with respect to $\M^{(n)}$. Therefore, we have $$\ai '(\M^{(n)})
=-\inf\{i\in\Z\mid \w_{ni}\neq\K_M\}.$$

The number $\ai '(\M)$ seems easier to handle than the $a$-invariant $\ai (\G(\M))$. For instance, we get the following formula with a simple proof. This is an analog of the formula about $\ai (\G(\M))$ due to \cite[Theorem 4.1 and Corollary 4.3]{HZ}.
For each $r\in\mathbb{R}$, we define $[r]=\max\{i\in\Z\mid i\le r\}$.

\begin{prop}\label{formula}
Let $n$ be a positive integer. Then the equality
$$\ai '(\M^{(n)})
=\left[\frac{\ai '(\M)}{n}\right]$$
holds true.
Hence, $\mathrm{a'}(\M)+1\le  n\big(\mathrm{a'}(\M^{(n)})+1\big)$, and furthermore the equality $$\mathrm{a'}(\M)+1= n\big(\mathrm{a'}(\M^{(n)})+1\big)$$
holds true whenever $\mathrm{a'}(\M)+1\in n\Z$. 
\end{prop}

\begin{proof}
For each integer $i$, we recall that $\w_{-i}\neq\K_M$ if and only if $i\le\ai '(\M)$
by the definition of $\ai '(\M)$, so that  $\w_{n(-i)}\neq\K_M$ if and only if $ni\le \ai '(\M)$. Thus, we obtain that
$$\ai '(\M^{(n)})
=\sup\left\{i\in \Z\ \bigg|\ i\le \frac{\ai '(\M)}{n}\right\},$$ as desired. 

Put $c=n(\ai '(\M^{(n)})+1)-(\ai '(\M)+1)$. By the obtained equality above, we have $\ai '(\M)/n<\ai '(\M^{(n)})+1$, and thus we see $c\ge 0$, whence $\mathrm{a'}(\M)+1\le  n\big(\mathrm{a'}(\M^{(n)})+1\big)$, as desired. 
Furthermore, suppose that $\mathrm{a'}(\M)+1\in n\Z$, and then $c\in n\Z$. By the obtained equality above, we have $\ai '(\M^{(n)})\le \ai '(\M)/n$, and therefore we see $c\le n-1$. Since $c\in n\Z$, this means $c=0$, whence $\mathrm{a'}(\M)+1= n\big(\mathrm{a'}(\M^{(n)})+1\big)$, as required.
\end{proof}

For each $\F$-filtration $\calN=\{N_i\}_{i\in\Z}$ of an $A$-module of $N$ and for each $\p\in \Spec A$, we set $\calN_\p=\{(N_i)_\p\}_{i\in\Z}$, which is an $\F_\p$-filtration of $N_\p$. Then $\R'(\calN)_\p\cong\R'(\calN_\p)$ as a graded $\R(\F_\p)$-module (recall that $\R'(\calN)\cong\bigoplus_{i\in\Z}N_i$ as graded $\R(\F)$-module). Let us notice the following.

\begin{lem}\label{supp1}
The equality $\Supp_A\K_{\R'(\M)}=\Supp_A\K_M$ holds, and hence $\w_\p=\{(\w_i)_\p\}_{i\in\Z}$ is a canonical $\F_\p$-filtration of $\K_{M_\p}$ with respect to $\M_\p$
for all $\p\in\Supp_A\K_M$.
\end{lem}

\begin{proof}
Take any $\p\in\Supp_A\K_{\R'(\M)}$. Since $\K_{\R'(\M)}\cong\R'(\w)$ as a graded $\R'(\w)$-module,  there exists an integer $i$ such that $(\w_i)_\p\neq 0$. Since $\w_i\subseteq \K_M$, we see $(\K_M)_\p\neq 0$, whence $\Supp_A\K_{\R'(\M)}\subseteq\Supp_A\K_M$. The converse is clear, as $\w_j=\K_M$ for all integers $j\ll 0$.
 
Let us check the last assertion. We take any $\p\in\Supp_A\K_M$. From $\p\in \Supp_A\K_{\R'(\M)}$, it follows from \cite[(1.6)]{A} that $(\K_{\R'(\M)})_\p\cong\K_{(\R'(\M)_\p)}$ as a graded $\R'(\F)_\p$-module, whence $\R'(\w_\p)\cong\K_{\R'(\M_\p)}$ as a graded $\R'(\F_\p)$-module. This completes the proof.
\end{proof}

When $\p\in \Supp_AM\cap\mathrm{V}(F_1)$, we have $\ai '(\M_\p)\in\Z$ because $M_\p\neq 0$ and $(F_1)_\p\neq A_\p$. Furthermore, we get the following.

\begin{lem}\label{supp2}
The inequality $\ai'(\M_\p)\le \ai'(\M)$ holds for all $\p\in \Supp_A\K_M\cap\mathrm{V}(F_1)$.
\end{lem}

\begin{proof}
Take any $\p\in\Supp_A\K_M\cap\mathrm{V}(F_1)$. It follows form Lemma \ref{supp1} that
$\w_\p=\{(\w_i)_\p\}_{i\in\Z}$ is a canonical $\F_\p$-filtration of $\K_{M_\p}$ with respect to $\M_\p$. 
Thus, $\ai '(\M_\p)\le\ai '(\M)$ by the definition of $\ai '(-)$.
\end{proof}

We define $\A(\M)=\bigcup_{i\ge 1}\Ass_AM/F_iM$. Put $\A=\A(\M)$ for simplicity. Then $$\A=\Ass_A\G(\M)=\{Q\cap A\mid P\in\Ass\G(\M)\},$$ which is a finite set (cf. \cite[(3)]{B} and \cite[Lemma 3]{ME}, and see also \cite[Exercise 6.7]{M}).
We notice that $\A\subseteq \mathrm{V}(F_1)$. 
With this notation, we obtain the following, which parallels the result \cite[Proposition 4.4]{HHK2}.

\begin{lem}\label{max}
$\ai '(\M)=\max\{\ai '(\M_\p)\mid \p\in\A\cap\Supp_A\K_M\}$.
\end{lem}

\begin{proof}
By Lemma \ref{supp2}, it suffices to prove that there exists $\p\in\A\cap\Supp_A\K_M$ such that $\ai '(\M)=\ai '(\M_\p)$, which means that
$\K_{M_\p}\neq (\w_{-\ai' (\M)})_\p$ because $\w_\p=\{(\w_i)_\p\}_{i\in\Z}$ is a canonical $\F_\p$-filtration of $\K_{M_\p}$ with respect to $\M_\p$ by Lemma \ref{supp1}.
Since 
$$
\Ass_{\G(\F)}\K_{\G(\M)}=\{P\in \Supp_{ \G(\F)}\G(\M)\mid \dim \G(\F)/P=\dim \G(\M)\},$$
it is observed that $\Ass_{\G(\F)}\K_{\G(\M)}\subseteq \Ass_{\G(\F)}\G(\M)$, and therefore $\Ass_{A}\K_{\G(\M)}\subseteq \A$.
The monomorphism $\K_{\R'(\M)}/t^{-1} \K_{\R'(\M)}(-1)\hookrightarrow \K_{\G(\M)}$ of graded $\G(\F)$-modules implies that $\Ass_A\K_{\R'(\M)}/t^{-1} \K_{\R'(\M)}\subseteq \Ass_{A}\K_{\G(\M)}$, and thus
$$\Ass_A\K_{\R'(\M)}/t^{-1} \K_{\R'(\M)}\subseteq \A\cap\Supp_A\K_M.$$
We recall that $[\K_{\R'(\M)}/t^{-1} \K_{\R'(\M)}]_{-\ai' (\M)-1}\cong \K_{M}/\w_{-\ai' (\M)}\neq 0$ as an $A$-module. Hence, taking $\p\in\Ass_A\K_{M}/\w_{-\ai' (\M)}$, we see $\p\in\A\cap\Supp_A\K_M$ and $\K_{M_\p}\supsetneq (\w_{-\ai' (\M)})_\p$. 
\end{proof}

The following proposition plays a key role in this paper.

\begin{prop}\label{mainthm}
Suppose that $[\H^{\dim M-1}_\fkM(\G(\M))]_{j}=0$ for all $j<\ai(\G(\M))$. Then $$\K_{\G(\M)}\cong \K_{\R'(\M)}/t^{-1}\K_{\R'(\M)}(-1)$$ as a graded $\G(\F)$-module. Hence,
$\K_{\G(\M)}\cong\G (\w)(-1)$ as a graded $\G(\F)$-module and the equality $\ai '(\M)= \ai (\G(\M))$ holds.
\end{prop}

\begin{proof}
Put $s=\dim M$, $R'=\R'(\M)$, and $G=\G(\M)$. Taking the graded local cohomology functor $\H^{i}_{\fkM'}(-)$ of the short exact sequence 
$
0\xrightarrow{}R'(1)\xrightarrow{t^{-1}}R'\xrightarrow{}G\xrightarrow{}0
$
of finitely generated graded $\R'(\F)$-modules induced by multiplication by $t^{-1}$ on $R'$,
 we get the exact sequence 
$$
\H^{s-1}_{\fkM'}(G)\xrightarrow{}\H^{s}_{\fkM'}(R')(1)\xrightarrow{\phi}\H^{s}_{\fkM'}(R')\xrightarrow{}\H^{s}_{\fkM'}(G)\xrightarrow{}\H^{s+1}_{\fkM'}(R')(1)\xrightarrow{}\H^{s+1}_{\fkM'}(R')\xrightarrow{} 0
$$
of graded $\R'(\F)$-module. 
In order to obtain the required isomorphism, it is enough to show the homomorphism $\phi$ above is surjective by the graded local duality theorem. We note that $\phi$ is multiplication by $t^{-1}$. We denote by $[\phi]_j$ the $j^{\th}$ homogeneous component of $\phi$ for every $j\in\Z$.

Let $i\in\Z$. We will show that the $A$-homomorphism $[\phi]_i: [\H^{s}_{\fkM'}(R')]_{i+1}\to[\H^{s}_{\fkM'}(R')]_i$ is surjective.
When $i>\ai(G)$, $[\phi]_i$ is surjective by the definition of $\ai(G)$ (recall that $\H^{s}_{\fkM'}(G)\cong \H^{s}_{\fkM}(G)$ as a graded $\R'(\F)$-module).

Suppose $i \le \ai(G)$ and take any element $h\in [\H^{s}_{\fkM'}(R')]_i$. Since $\H^{s}_{\fkM'}(R')$ is an Artinian $\R'(\F)$-module, the descending chain
$
\H^{s}_{\fkM'}(R')\supseteq t^{-1}\H^{s}_{\fkM'}(R')\supseteq t^{-2}\H^{s}_{\fkM'}(R')\supseteq\dots
$
is stationary, i.e., there exists a nonnegative integer $k$ such that $$t^{-k}\H^{s}_{\fkM'}(R')= t^{-k-1}\H^{s}_{\fkM'}(R'),$$
and hence we write $t^{-k}h=t^{-k-1}h'$ for some $h'\in [\H^{s}_{\fkM'}(R')]_{i+1}$. We notice that, for all integers $j<\ai(\G(\M))$, it follows from $[\H^{s-1}_\fkM(\G(\M))]_{j}=0$ that $[\phi]_{j}$ is a monomorphism of $A$-modules.
Let $\psi: [\H^{s}_{\fkM'}(R')]_{i}\to[\H^{s}_{\fkM'}(R')]_{i-k}$ denote the composite of monomorphisms:
$$
[\H^{s}_{\fkM'}(R')]_{i}\xrightarrow{[\phi]_{i-1}}[\H^{s}_{\fkM'}(R')]_{i-1}\xrightarrow{[\phi]_{i-2}}[\H^{s}_{\fkM'}(R')]_{i-2} \xrightarrow{[\phi]_{i-3}}\cdots\xrightarrow{[\phi]_{i-k}}[\H^{s}_{\fkM'}(R')]_{i-k}
$$
of $A$-modules. Then, $\psi$ is injective. Since $\phi$ is multiplication by $t^{-1}$, we have $\psi(x)=t^{-k}x$ for each $x\in[\H^{s}_{\fkM'}(R')]_{i}$, so that $\psi(h)=\psi(t^{-1}h')$. Thus, $h=t^{-1}h'$, as $\psi$ is injective. This means that $\phi$ is surjective, as desired. 

Let us check the last assertion. Since $\K_{\R'}\cong \R'(\w)$ as a graded $\R'(\F)$-module, we have $\K_{\R'}/t^{-1}\K_{\R'}\cong\G(\w)$ as a graded $\G(\F)$-module.
Therefore, $\K_{G}\cong\G(\w)(-1)$ as a graded $\G(\F)$-module and the equality $\ai '(\M)= \ai (\G(\M))$ holds true.
\end{proof}

When $M=A$ (and then $\M=\F$), let us call an extended canonical $\F$-filtration of $\K_M$ with respect to $\M$ {\it an extended canonical $\F$-filtration of $\K_A$} for short (that is, we remove the sentence ``with respect to $\M$'').
Let us note the next result essentially due to \cite[Theorem 2.1]{GI}.

\begin{lem}[cf. \cite{GI}]\label{unique}
Assume that the ring $A$ satisfies the condition $(S_2)$ of Serre. Then an extended canonical $\F$-filtration of $\K_A$ is uniquely determined. 
\end{lem}

\begin{proof}
Let both of $\w=\{\w_n\}_{n\in\Z}$ and $\w'=\{\w'_n\}_{n\in\Z}$ be canonical $\F$-filtrations of $\K_A$. Then we have an isomorphism $\phi : \R'(\w)\xrightarrow{}\R'(\w')$ of a graded $\R'(\F)$-module. There exists an integer $n$ such that $\w_n=\w'_n=\K_A$, so that the isomorphism $\phi |_{n}:\K_A\xrightarrow{\sim}\K_A$ of $A$-modules is multiplication by a unit in $A$ because the ring $A$ satisfies the condition $(S_2)$ of Serre. By Lemma \ref{restriction1}, the isomorphism $\phi|_i: \w_i\xrightarrow{\sim}\w'_i$ is also multiplication by the unit in $A$ for every integer $i\in\Z$. This means that $\w=\w'$, as desired. 
\end{proof}

In the rest of this section, let us consider the case where $M=A$. 
Let $\w=\{\w_i\}_{i\in\Z}$ be an extended canonical $\F$-filtration of $\K_A$.
We are now ready to prove Theorem \ref{main2}.

\begin{proof}[Proof of Theorem \ref{main2}]
We put $\A=\A(\F)$ for simplicity.
Since $\K_A\cong A$ as an $A$-module, the ring $A$ satisfies the Serre's condition $(S_2)$, and we have $\A\subseteq \Spec A=\Supp_A\K_A$. Hence, $\A\subseteq \Supp_A\K_{\R'(\F)}$ by Lemma \ref{supp1}, so that
$(\K_{\R'(\F)})_\p\cong \K_{\R'(\F_\p)}$ as a graded $\R'(\F_\p)$-module for all $\p\in\A$. Thus, the implication $(1)\Rightarrow (2)$ holds true.

Let us show the implication $(2)\Rightarrow (1)$.
Take any $\p\in\A$. We will initially determine the number $\ai'(\F)$.
We note that
$\w_\p=\{(\w_i)_\p\}_{i\in\Z}$ is the extended canonical $\F_\p$-filtration of $A_\p$ by Lemma \ref{supp1}.
Since $\K_{\R'(\F_\p)}\cong\R'(\F_\p)(k)\cong\R'(\F_\p(k))$ as a graded $\R'(\F_\p)$-module, it follows from Lemma \ref{unique} that $\w_\p=\F_\p(k)$, that is, 
$
(\w_i)_\p=(F_{i+k})_\p
$
for all $i\in\Z$,
and hence $\ai'(\F_\p)=k-1$ because $(F_1)_\p\neq A_\p$ (recall that $\A\subseteq \mathrm{V}(F_1)$). Therefore, the equality $\ai'(\F)=k-1$ holds true by Proposition \ref{max}. 

Thus, we obtain the equality 
$
\w_{-k}=A,
$ 
and then $F_{i+k}\subseteq \w_{i}$ for all $i\in\Z$ because $F_{i+k}\, \w_{-k}\subseteq \w_{i}$.
Therefore, $\F(k)\subseteq \w$, which yields the inclusion map $$\iota: \R'(\F(k))\hookrightarrow \R'(\w),$$ and then $[\coker \iota]_i\cong \w_i/F_{i+k}$ as an $A$-module for all $i\in\Z$. Thus, we obtain that 
$
\Ass_A\coker\iota\subseteq \A
$
because $\w_i/F_{i+k}\subseteq A/F_{i+k}$ for all $i\in\Z$.

Suppose that $\coker\iota\neq 0$ and we seek a contradiction. Take any $\p'\in\Ass_A\coker\iota$, and then there exists an integer $i$ such that $(\w_i/F_{i+k})_{\p'}\neq 0$. Since $\p'\in \A$, this contradicts to the equality $(\w_i)_{\p'}=(F_{i+k})_{\p'}$ as obtained above. Thus, we observe that $\coker\iota= 0$, so that $\R'(\F(k))=\R'(\w)$.
Since $\R'(\F(k))\cong\R'(\F)(k)$ and $\R'(\w)\cong \K_{\R'(\F)}$, we obtain that $\K_{\R'(\F)}\cong\R'(\F)(k)$ as a graded $\R'(\F)$-module, as desired.
\end{proof}

Let us notice the following.

\begin{rem}\label{TI} 
Assume that $\R(\F)$ is Cohen-Macaulay with 
$\dim \R(\F)=d+1$. Then 
$\ai(\G(\F))<0$ and the graded local cohomology module $\H^{d-1}_\fkM(\G(\F))$ is concentrated in degree $-1$ by a generalized Trung-Ikeda's theorem in \cite{GN} and \cite{TVZ}. Therefore,  
the assumption that $[\H^{d-1}_\fkM(\G(\F))]_{j}=0$ for all integers $j<\ai(\G(\F))$ in Theorem \ref{mainthm} is always satisfied.
\end{rem}

As a direct consequence of Theorem \ref{mainthm}, we get the next result by the remark above.

\begin{cor}\label{maincor}
Assume that $\R(\F)$ is Cohen-Macaulay with 
$\dim \R(\F)=d+1$. Then $$\K_{\G(\F)}\cong \left(\K_{\R'(\F)}\big/t^{-1}\K_{\R'(\F)}\right)(-1)$$ as a graded $\G(\F)$-module.
Hence,
$\K_{\G(\F)}\cong\G (\w)(-1)$ as a graded $\G(\F)$-module and the equality $\ai '(\F)= \ai (\G(\F))$ holds.
\end{cor}

It is observed that $\F$ and $\w$ are strongly separated (see, e.g., \cite[Lemma 2.1]{GI}), and we obtain the following result, which will be 
instrumental of a proof of Theorem \ref{main1}.

\begin{prop}\label{R'inj}
Let $a$ be an integer. Suppose that $[\H^{d-1}_\fkM(\G(\F))]_{j}=0$ for all integers $j<\ai(\G(\F))$. Then the following three conditions are equivalent.\vspace{0mm}
\begin{enumerate}
	\item $[\K_{\R'(\F)}]_{\ge -a-1}\cong\R(\F)(a+1)$ as a graded $\R(\F)$-module.\vspace{1mm}
	\item $[\K_{\G(\F)}]_{\ge -a}\cong\G(\F)(a)$ as a graded $\G(\F)$-module. \vspace{1mm}	
	\item
\begin{enumerate}
\item[{\rm (i)}] $[\K_{\R'(\F)}]_{-a-1}\cong A$ as an $A$-module and 
\item[{\rm (ii)}]  there exists a monomorphism $\G(\F)(a)\hookrightarrow \K_{\G(\F)}$ of graded $\G(\F)$-modules. 
\end{enumerate}
\end{enumerate}
\end{prop}

\begin{proof}
It follows from $\K_{\R'(\F)}\cong\R'(\w)$ as a graded $\R'(\F)$-module that 
the condition (1) is equivalent to saying that $\R(\w)_{\ge -a-1}\cong\R(\F)(a+1)$ as a graded $\R(\F)$-module.

From Theorem \ref{mainthm}, it follows that $\K_{\G(\F)}\cong\G (\w)(-1)$ as a graded $\G(\F)$-module. Therefore, the condition (2) if and only if $\G(\w)_{\ge -a-1}\cong\G(\F)(a+1)$ as a graded $\G(\F)$-module. 

Since $[\K_{\R'(\F)}]_{-a-1}\cong \w_{-a-1}$ as an $A$-module,
the condition (i) if and only if $\w_{-a-1}\cong A$ as an $A$-module. 
Furthermore, the condition (ii) is equivalent to saying that there exists a monomorphism $\G(\F)(a+1)\hookrightarrow \G(\w)$ of graded $\G(\F)$-modules by Corollary \ref{maincor}.

Thus, the proposition directly follows from Proposition \ref{embed2}.
\end{proof}

We are now ready to prove Theorem \ref{main1}.

\begin{proof}[Proof of Theorem \ref{main1}]
Thanks to \cite[Theorem 1.1]{TVZ}, the condition (1) in Theorem \ref{main1} is equivalent to saying that 
$[\K_{\G(\F)}]_{\ge 2}\cong\G(\F)(-2)$ as a graded $\G(\F)$-module. Therefore, it follows from Proposition \ref{R'inj} together with Remark \ref{TI} that the conditions (1), (2), and (3) in Theorem \ref{main1} are equivalent to each other.
\end{proof}

The following result is concerned with the relationship of quasi-Gorenstein properties between $\R'(\F)$ and $\G(\F)$.

\begin{prop}\label{qGR'} 
Let $a$ be an integer. 
Suppose that $[\H^{d-1}_\fkM(\G(\F))]_{j}=0$ for all integers $j<\ai(\G(\F))$. Then the following three conditions are equivalent.\vspace{0mm}
\begin{enumerate}
	\item $\K_{\R'(\F)}\cong\R'(\F)(a+1)$ as a graded $\R'(\F)$-module.\vspace{1mm}
	\item $\K_{\G(\F)}\cong\G(\F)(a)$ as a graded $\G(\F)$-module. \vspace{1mm}
	\item 
\begin{enumerate}
\item[{\rm (i)}] $\K_A\cong A$ as an $A$-module,  
\item[{\rm (ii)}] $a=\ai(\G(\F))$, and 
\item[{\rm (iii)}] there exists a monomorphism $\G(\F)(a)\hookrightarrow \K_{\G(\F)}$ of graded $\G(\F)$-modules. 
\end{enumerate}
\end{enumerate}
\end{prop}

\begin{proof}
$(1)\Rightarrow (2)$ Applying the functor $\G(\F)\otimes_{\R'(\F)}-$ 
to the both sides of the given isomorphism $\K_{\R'(\F)}\cong\R'(\F)(a+1)$ of graded $\R'(\F)$-modules, we obtain that $$\K_{\R'(\F)}/t^{-1}\K_{\R'(\F)}\cong\G(\F)(a+1)$$ as a graded $\G(\F)$-module. Then it follows from Theorem \ref{mainthm} that $\K_{\G(\F)}\cong\G(\F)(a)$ as a graded $\G(\F)$-module.

$(2)\Rightarrow (3)$ 
Since $\K_{\G(\F)}\cong\G(\F)(a)$ as a graded $\G(\F)$-module, the conditions (ii) and (iii) clearly hold true.
Since $[\K_{\G(\F)}]_{\ge -a}\cong\G(\F)(a)$ as a graded $\G(\F)$-module, we have 
$[\K_{\R'(\F)}]_{-a-1}\cong A$ by Proposition \ref{R'inj}. This implies that $\K_A\cong A$ as an $A$-module because $[\K_{\R'(\F)}]_{-a-1}\cong \K_A$ as $A$-module by Lemma \ref{-a-1}.

$(3)\Rightarrow (1)$ 
Since $\K_A\cong A$ as an $A$-module,
we can take $\w=\{\w_i\}_{i\in\Z}$ as an extended canonical $\F$-filtration of $A\, $. By using Lemma \ref{-a-1}, we see the equality $\w_{-a-1}=A$ because $a=\ai(\G(\F))$, and hence $[\K_{\R'(\F)}]_{-a-1}\cong A$ as $A$-modules.
Then, it follows from Proposition \ref{R'inj} that $$[\K_{\R'(\F)}]_{\ge -a-1}\cong\R(\F)(a+1)$$
as a graded $\R(\F)$-module.
By Lemma \ref{y}, there is $y\in A$ such that $0:_Ay=0$ and $\w_{i}=F_{i+a+1} y$ for every integer $i\ge -a-1$. Since $\w_{-a-1}=A$, we have $A=A y$, which means that the element $y$ is a unit of A, and hence, $\w_{i}=F_{i+a+1}$ for every integer $i\ge -a-1$. Therefore, the equality $\w=\F(a+1)$ holds true because $\w_i=F_{i+a+1}=A$ for all integers $i< -a-1$. Thus, $\K_{\R'(\F)}\cong\R'(\F)(a+1)$ as a graded $\R'(\F)$-module, since $\K_{\R'(\F)}\cong\R'(\w)=\R(\F(a+1))\cong\R(\F)(a+1)$ as a graded $\R'(\F)$-module.
\end{proof}

The next theorem directly follows from Proposition \ref{qGR'} together with Remark \ref{TI}.

\begin{thm}\label{corqGR'} 
Let $a$ be an integer and assume that $\R(\F)$ is a Cohen-Macaulay ring with 
$\dim \R(\F)=\dim A+1$. Then the following three conditions are equivalent.\vspace{1mm}
\begin{enumerate}
	\item $\K_{\R'(\F)}\cong\R'(\F)(a+1)$ as a graded $\R'(\F)$-module.\vspace{2mm}
	\item $\K_{\G(\F)}\cong\G(\F)(a)$ as a graded $\G(\F)$-module. \vspace{2mm}
	\item 
\begin{enumerate}
\item[{\rm (i)}] $\K_A\cong A$ as an $A$-module,  
\item[{\rm (ii)}] $a=\ai(\G(\F))$, and 
\item[{\rm (iii)}] there exists a monomorphism $\G(\F)(a)\hookrightarrow \K_{\G(\F)}$ of graded $\G(\F)$-modules. 
\end{enumerate}
\end{enumerate}
\end{thm}

We close this section with the proof of Corollaries \ref{maincor1} and \ref{maincor2}.

\begin{proof}[Proof of Corollary \ref{maincor1}]
$(1)\Leftrightarrow (3)$ Since $\R(\F)$ is Cohen-Macaulay and $\grade F_1\ge 2$, $\K_{\R(\F)}\cong\R(\F)(-1)$ as a graded $\R(\F)$-module if and only if 
$\K_{\G(\F)}\cong\G(\F)(-2)$ as a graded $\G(\F)$-module by \cite[Corollary 3.4]{TVZ}. The latter condition is equivalent to saying that the condition (3) by Theorem \ref{corqGR'}.

$(2)\Leftrightarrow (3)$ This directly follows from Theorem \ref{corqGR'}.
\end{proof}

\begin{proof}[Proof of Corollary \ref{maincor2}]
The required equivalence directly follows from Theorem \ref{main2} together with Corollary \ref{maincor1}.
\end{proof}

\section{Gorenstein Rees algebras of powers of ideals}

This section will investigate the Gorenstein Rees algebras of powers of ideals and give the proof of Theorem \ref{in}, Corollary \ref{incor}, and Example \ref{qGex}.
First of all, let us give a proof of the assertions $(1)$ and $(2)$ of Example \ref{qGex}. The assertions $(3)$ will be shown at the end of this paper. For each $\Z$-graded ring $S$ with a unique homogeneous maximal ideal $\fkN$, let $\depth S$ stand for the depth of $S$ relative to $\fkN$, namely $\depth S:=\depth S_\fkN$.

\begin{proof}[Proof of the assertions $(1)$ and $(2)$ in Example \ref{qGex}]
(1) We will show $\grade \G(I)_+>0$. Take any $P\in\Spec \G(I)$ such that $\G(I)_+\subseteq P$. Suppose that $\depth \G(I)_P=0$, and we seek a contradiction. Put $\p=P\cap A$, and then $I\subseteq \p$. Since $\G(I)_+\subseteq P$, we have $\depth \G(I)_\p=\depth \G(I)_P$. Hence, $\depth \G(I_\p)=0$, since $\G(I)_\p\cong\G(I_\p)$ as rings. Suppose that $\p\neq\m$. We have $I_\p=(x_2, x_3, \dots, x_d)_\p$, since $x_1\not\in \p$. Hence, $I_\p$ is a parameter ideal of the Cohen-Macaulay local ring $A_\p$ because $\NCM(A)\le 0$, and then $\depth \G(I_\p)=d-1$ because the ring $\G(I_\p)$ is isomorphic to a polynomial ring in $d-1$ variables over $A_\p/I_\p$.
This contradicts to $\depth \G(I_\p)=0$. Thus, we obtain that $\p=\m$. 

By using \cite[Theorem 2.4]{K2}, we observe $\depth\R(I)\ge 4$, since $\depth A\ge 2$. We have $\depth\G(I)=\depth\G(I_\p)=0$, as $\p=\m$. This is a contradiction to the equality $\depth\R(I)=\depth\G(I)+1$ given by \cite[Theorem 3.10]{HM}. Thus, $\grade \G(I)_+>0$, as required.
Furthermore, thanks to \cite[Theorem (12.3)]{Mc2}, we obtain that ${(I^i)}^\ast=I^i$ for all integers $i\in\Z$.

(2) Let us initially show that $\Ass \G(I)=\Min\G(I)$. Take any $P\in\Ass \G(I)$, and then $\G(I)_+\not\subseteq P$ because $\grade \G(I)_+>0$. Thanks to \cite[Theorem 3.1]{K2}, we have $\R(I^{n})$ is a Cohen-Macaulay ring for some positive integer $n$.  Identifying $\G(I)=\R(I)/I\R(I)$, we write $P=Q/I\R(I)$ for some $Q\in\Spec \R(I)$ satisfying $Q\supseteq I\R(I)$. Since $\R(I)_+\not\subseteq Q$, the ring $\R(I)_Q$ is Cohen-Macaulay because so is the ring $\R(I^n)$ (recall that $\Proj\R(I)\cong \Proj\R(I^n)$ as a scheme).
We can find a nonzero divisor $x\in I$ on $A$ such that $xt\not\in Q$, and then we have $\R(I)_{xt}/x\R(I)_{xt}\cong \G(I)_{xt}$ as a ring, so that $\R(I)_Q/x\R(I)_Q\cong \G(I)_Q$ as a ring. Therefore, the ring $\G(I)_P\ (=\G(I)_Q)$ is Cohen-Macaulay. Since $\depth \G(I)_P=0$, we obtain that $P\in\Min \G(I)$, and thus $\Ass \G(I)=\Min\G(I)$, as desired. 

We set $\A(I)=\bigcup_{i\ge 0}\Ass_AA/I^i$, and then $\A(I)=\Ass_A \G(I)$.
The ring $\G(I)$ is quasi-unmixed because so is the ring $A$ (see \cite[Corollary 18.24]{HIO}). 
Therefore, 
all associated prime ideals of $\G(I)$ have same codimension. Hence, we obtain that $$\A(I)=\{\p\in\V(I)\mid \dim A_\p=\ell (I_\p)\}$$ (cf. \cite[Proposition 4.1]{Mc1} and see also \cite[Lemma 3.1]{GI}). 
Put $J=(x_2, x_3, \dots, x_d)$. Since the ideal $J$ is generated by a subsystem of parameters for $A$ of $d-1$ elements, we have $\G(J)/\m\G(J)$ is isomorphic to a polynomial ring in $d-1$ variables over $A/\m$, and hence $\dim \G(J)/\m\G(J)=d-1$. Thanks to \cite[(3.2.1)]{K1}, we obtain that the ideal $J$ is a reduction of $I$, and hence the ring $\G(I)/\m\G(I)$ is a finitely generated $\G(J)/\m\G(J)$-module  by \cite[Theorem 8.2.1]{SH}. Therefore, $\dim \G(I)/\m\G(I)\le d-1$, whence $\ell (I)=d-1$ by \cite[Proposition 8.3.9]{SH}. Thus, $\m\not\in\A(I)$.

Take any $\p\in\A(I)$. Since $\p\neq\m$, we have $\dim A_\p =d-1$, and hence $I_\p =J_\p$, as $x_1\not\in \p$. Therefore, $I_\p$ is a parameter ideal of $A_\p$. In addition, the ring $A_\p$ is Gorenstein, because the ring $A$ is quasi-Gorenstein and $\dim\NCM (A)\le 0$. 
Thus, $\G(I_\p)$ is isomorphic to a polynomial ring in $d-1$ variables over the Gorenstein ring $A_\p/I_\p$, and hence the ring $\G(I_\p)$ is Gorenstein with $\ai(\G(I_\p))=-d+1$. Thanks to \cite[Proposition 4.4]{HHK2}, we obtain that $\ai(\G(I))=-d+1$, as required.

Moreover, the ring $\R'(I_\p)$ is also Gorenstein, since so is the ring $\G(I_\p)$. Therefore, 
$$\K_{\R'(I_\p)}\cong\R'(I_\p)(k)$$ as a graded $\R'(I_\p)$-module for some $k\in\Z$ by \cite[Proposition 3.6.11]{BH}. Then $k=-d+2$ because $\K_{\G(I_\p)}\cong \left(\K_{\R'(I_\p)}\big/t^{-1}\K_{\R'(I_\p)}\right)(-1)\cong\G(I_\p)(k-1)$ and $\K_{\G(I_\p)}\cong\G(I_\p)(-d+1)$ as graded $\G(I_\p)$-module (recall that $\G(I_\p)$ is isomorphic to a polynomial ring in $d-1$ variables).
This means that the number $k$ does not depends on the choice of $\p\in\A(I)$. 
Thus, we obtain from Theorem \ref{main2} that $$\K_{\R'(I)}\cong\R'(I)(-d+2)$$ as a graded $\R'(I)$-module, as desired.
Furthermore, let $m$ and $n$ be positive integers such that $mn=d-2$. Applying the $n^{\th}$ Veronese functor $[-]^{(n)}$ of the isomorphism above, we write $[\K_{\R'(I_\p)}]^{(n)}\cong[\R'(I_\p)(-d+2)]^{(n)}$ as a graded $\R'(I_\p)^{(n)}$-module, whence 
$\K_{\R'(I^n_\p)}\cong\R'(I^n_\p)(-m)$ as a graded $\R'(I^n_\p)$-module, because $mn=d-2$. 
\end{proof}

In the rest of this section, let $(A, \m)$ be a Noetherian local ring with $d=\dim A>0$ and assume that the ring $A$ is a homomorphic image of a Gorenstein local ring. 
Let $I$ be an ideal of $A$ such that $I\neq A$ and $I$ contains a regular element on $A$. Hence, $\dim\R(I)=d+1$. We consider the case where $\F=\{I^i\}_{i\in\mathbb{Z}},$ and continue to use the notation introduced in the previous section unless otherwise specified. 
Let $\w=\{\w_i\}_{i\in\mathbb{Z}}$ denote a canonical $\F$-filtration of $\K_A$. Hence, $\K_{\R'(I)}\cong\R'(\w)$ as a graded $\R'(I)$-module.

For each ideal $J$ of $A$, we set $J^n=A$ for each integer $n\le 0$, and define 
$$J^\ast=\bigcup_{i\ge 0}J^{i+1}:_AJ^i,$$
which is called the Ratliff-Rush closure of $J$
(\cite{RR} and \cite[Chapter 11]{Mc2} for details). We have $I^i\subseteq (I^i)^\ast$ and ${(I^i)}^\ast {(I^j)}^\ast\subseteq {(I^{i+j})}^\ast$ for all $i, j\in\mathbb{Z}$, and moreover ${(I^i)}^\ast=I^i$ for all integers $i\gg 0$.
Put $$\F^\ast=\{(I^i)^\ast\}_{i\in\mathbb{Z}},$$ which forms a filtration of ideals in $A$ such that $\F\subseteq\F^\ast$. 
Since $\R'(I)$ is a graded subring of $\R'(\F^\ast)$, we have 
the short exact sequence 
$$
0\to \R'(I)\to \R'(\F^\ast)\to \R'(\F^\ast)/\R'(I) \to 0
$$
of graded $\R'(I)$-modules. For each $i\in\Z$, it is noticed that $[\R'(\F^\ast)/\R'(I)]_i\cong (I^i)^\ast/I^i$ as an $A$-module and that $\dim \big((I^i)^\ast/I^i\big)<d$, since the ideal $I$ contains a nonzero divisor on $A$.
Since $I^i= (I^i)^\ast$ for all integers $i\gg 0$,  there exists an integer $r$ such that $[\R'(\F^\ast)/\R'(I)]_i=0$ whenever $i\in\Z\setminus\{1,2,\dots ,r\}$. 
Hence, $\R'(\F^\ast)$ is a finitely generated $\R'(I)$-module with $\dim \R'(\F^\ast)/\R'(I)< d$. Since $\dim \R'(I)=\dim \R'(\F^\ast)=d+1$, 
we obtain that $$\K_{\R'(\F^\ast)}\cong \K_{\R'(I)}$$ as a graded $\R'(I)$-module.

Let us state the following result, which illustrates that the $\F$-filtration $\w$ is also an extended canonical $\F^\ast$-filtration of $\K_A$.

\begin{lem}\label{ast}The following two assertions hold true.
\vspace{0mm}
\begin{enumerate}
\item $(I^i)^\ast\w_j\subseteq \w_{i+j}$ for all $i,j\in\Z$, and hence $\w$ is an $\F^\ast$-filtration of $\K_A$.\vspace{2mm}
\item $\K_{\R'(\F^\ast)}
\cong\R'(\w)$ as a graded $\R'(\F^\ast)$-module.
\end{enumerate}
\end{lem}

\begin{proof} 
Since $\K_{\R'(I)}\cong\R'(\w)$ as a graded $\R'(I)$-module, the above $\R'(I)$-isomorphism $\K_{\R'(\F^\ast)}\cong \K_{\R'(I)}$ implies an isomorphism $\R'(\w)\xrightarrow{\sim} \K_{\R'(\F^\ast)}$ of graded $\R'(I)$-modules.
Thus, the assertions (1) and (2) directly follows from Lemma \ref{extendF}.
\end{proof}

In the following discussion,  we will pay attention to the condition that $$[\K_{\R'(I^n)}]_{\ge -b-1}\cong\R(I^n)(b+1)$$ as a graded $\R(I^n)$-module for some positive integer $n$ and for some integer $b$. This is a necessary condition of the Gorenstein property of the Rees algebra $\R(I^n)$ in the case where $b=-2$ (see Theorem \ref{main1}).
Moreover, this condition is automatically satisfied when $\R'(I^n)$ is a quasi-Gorenstein ring. A typical example of the quasi-Gorenstein extended Rees algebras is stated as follows: when $A$ is a Gorenstein local ring of dimension $d\ge 2$ and $\q$ is a parameter ideal of $A$, we have $\K_{\R'(\q)}\cong\R'(\q)(-d+1)$ as a graded $\R'(\q)$-module. Hence, taking a positive integers $n$ and an integer $b$ such that $d-1=n(-b-1)$, we obtain that $\K_{\R'(\q^n)}\cong\R'(\q^n)(b+1)$ as a graded $\R'(\q^n)$-module, so that
$[\K_{\R'(\q^n)}]_{\ge -b-1}\cong\R(\q^n)(b+1)$ as a graded $\R(\q^n)$-module. For the case of a non-Gorenstein base ring $A$, see, e.g., Example \ref{qGex} (2).

We recall the notation $\calN^{(n)}=\{N_{ni}\}_{i\in\Z}$ for each $\F$-filtration $\calN=\{N_i\}_{i\in\Z}$ of an $A$-module $N$ and for each positive integer $n$.
Since $\F=\{I^i\}_{i\in\mathbb{Z}}$ in this section, we have $\F^{(n)}=\{I^{ni}\}_{i\in\mathbb{Z}}$,
and then $\R'(I^n)=\R'(\F^{(n)})\cong\R'(\F)^{(n)}=\R'(I)^{(n)}$ as a ring. In the next lemma, we will consider the number $$\ai '(\F) =-\inf\{i\in\Z\mid \w_i\neq\K_A\},$$ as introduced in Section 3.
Let us express
\begin{center}
$\ai '(I):=\ai '(\F)$\ \ and\ \ $\ai '(I^n):=\ai '(\F^{(n)})$
\end{center}
for simplicity. We recall that $\ai '(I^n)
=-\inf\{i\in\Z\mid \w_{ni}\neq\K_A\},$ as $\w^{(n)}
$ 
is an extended canonical $\F^{(n)}$-filtration of $\K_A$.

\begin{lem}\label{rr}
Let $n$ be a positive integer and assume that 
$[\K_{\R'(I^n)}]_{\ge -b-1}\cong\R(I^n)(b+1)$ as a graded $\R(I^n)$-module
for some integer $b$. Then there exists $y\in\K_A$ such that $$\w_{ni}=I^{n(i+b+1)}y$$ for all integers $i\ge -b-1$ and $0:_Ay=0$. Moreover, the following four assertions hold.\vspace{0mm}
\begin{enumerate}
\item $\w_{i}=(I^{i+n(b+1)})^\ast y$ for all integers $i\ge n(-b-1)$.\vspace{2mm}
\item $[\K_{\R'(\F^\ast)}]_{\ge n(-b-1)}\cong\R(\F^\ast)(n(b+1))$ as a graded $\R(\F^\ast)$-module.
\vspace{2mm}
\item $\mathrm{a'}(I^n)\ge b$, and hence $\mathrm{a}(\G(I^n))\ge b$.
\vspace{2mm}
\item $\mathrm{a'}(I)\ge n(b+1)-1$, and hence $\mathrm{a}(\G(I))\ge n(b+1)-1$.
\end{enumerate}
\end{lem}

\begin{proof}
Since $\R'(I^n)\cong \R'(I)^{(n)}$ as a graded ring, we have $$\K_{\R'(I^n)}\cong \K_{(\R'(I)^{(n)})}\cong (\K_{\R'(I)})^{(n)}\cong \R'(\w)^{(n)}\cong \R'(\w^{(n)})$$
 as a graded $\R'(I^n)$-module, so that $\R'(\w^{(n)})_{\ge -b-1}\cong\R(I^n)(b+1)$ as a graded $\R(I^n)$-module by the assumption. Hence, it directly follows from Lemma \ref{y} that there exists $y\in\K_A$ such that $0:_Ay=0$ and $\w_{ni}=I^{n(i+b+1)}y$ for every integer $i\ge -b-1$.
 
To prove the assertion (1), let us first show the following.

\begin{claim*}
There exists an integer $s$ such that
$\w_{i}=I^{i+n(b+1)}y$ for all integers $i\ge ns$. 
\end{claim*}

\begin{proof}
Since $\R'(\w)\cong \K_{\R'(I)}$ as a graded $\R'(I)$-module,  we have $\R'(\w)$ is a finitely generated $\R'(I)$-module, and hence there exists an integer $r$ such that $\w_{i+1}=I\w_{i}$ for all integers $i\ge r$. Choose an integer $s$ satisfying $s\ge -b-1$ and $ns\ge r$. Then, we obtain that $\w_{ns}=I^{n(s+b+1)}y$ and $\w_{i+1}=I\w_{i}$ for all integers $i\ge ns$. By using induction on $i$, we establish the equality $\w_{i}=I^{i+n(b+1)}y$ for every integer $i\ge ns$, as required.
\end{proof}

So, we can take an integer $s$ such that $\w_{i}=I^{i+n(b+1)}y$ for all integers $i\ge ns$.
Let $j$ be an integer such that $j\ge n(-b-1)$.
Then $\w_j\subseteq Ay$, as $\w_{n(-b-1)}=Ay$. 

Let us next show the inclusion
$$\w_{j}\subseteq (I^{j+n(b+1)})^\ast y$$ holds true. 
In fact,
we have the equality $(I^{j+n(b+1)})^\ast=\bigcup_{i\ge 0}I^{i+j+n(b+1)}:_AI^{i}$ by \cite[Proposition 11.1 (v)]{Mc2}, and hence there exists an integer $i$ such that $i+j\ge ns$ and $(I^{j+n(b+1)})^\ast=I^{i+j+n(b+1)}:_AI^{i}$. Since $I^i\w_{j}\subseteq \w_{i+j}=I^{i+j+n(b+1)}y$, we obtain $$\w_j\subseteq I^{i+j+n(b+1)}y:_{Ay}I^{i}=(I^{i+j+n(b+1)}:_AI^{i})y,$$ where the last equality follows from the equality $0:_Ay=0$. Thus, we obtain that $\w_{j}\subseteq (I^{j+n(b+1)})^\ast y$, as required.

Finally, we will show the converse inclusion
$$\w_{j}\supseteq (I^{j+n(b+1)})^\ast y$$ 
holds true. In fact, we obtain from Lemma \ref{ast} that $(I^{j+n(b+1)})^\ast \w_{n(-b-1)}\subseteq \w_{j}$, and then $(I^{j+n(b+1)})^\ast\cdot Ay\subseteq \w_{j}$ because $\w_{n(-b-1)}=Ay$. 
Thus, $(I^{j+n(b+1)})^\ast y\subseteq \w_{j}$, as required. This completes a proof of the assertion (1). 

(2)
By Lemma \ref{ast}, we have $\w$ is an extended canonical $\F^\ast$-filtration of $\K_A$, so that $\K_{\R'(\F^\ast)}\cong\R'(\w)$ as graded $\R'(\F^\ast)$-module. 
It follows form the assertion (1) together with Lemma \ref{y} that $\R'(\w)_{\ge n(-b-1)}\cong\R(\F^\ast)(n(b+1))$ as a graded $\R(\F^\ast)$-module. Thus, $[\K_{\R'(\F^\ast)}]_{\ge n(-b-1)}\cong\R(\F^\ast)(n(b+1))$ as a graded $\R(\F^\ast)$-module, as desired.

(3)  
Since $\w_{n(-b)}=I^ny\subsetneq Ay\subseteq \K_A$, we get $\w_{n(-b)}\neq\K_A$, and thus $\mathrm{a'}(I^n)\ge b$ by the definition of  $\ai'(I^n)$. The inequality $\mathrm{a}(\G(I^n))\ge b$ follows from Lemma \ref{-a-1}.

(4) By the assertion (1), we have $\w_{n(-b-1)+1}=(I^1)^\ast y\subsetneq Ay\subseteq \K_A$, and thus $\mathrm{a'}(I)\ge n(b+1)-1$. 
The inequality $\mathrm{a}(\G(I))\ge n(b+1)-1$ follows from Lemma \ref{-a-1}.
\end{proof}

When $\R(I^n)$ is a Cohen-Macaulay ring for some integer $n>0$, it follows from \cite[Proof of Theorem (3.1)]{I} that $\grade \G(I^n)_+>0$, and then $(I^{ni})^\ast=I^{ni}$ for all $i\in\Z$, whence we have $(\F^\ast)^{(n)}=\{I^{ni}\}_{i\in\Z}$. The proposition below presents that the converse implication of Lemma \ref{rr} (2) holds true when the ring $\R(I^n)$ is Cohen-Macaulay.

\begin{prop}\label{vero}
Let $b$ be an integer. Assume that $\R(I^n)$ is a Cohen-Macaulay ring for some integer $n>0$. Then 
the following three conditions are equivalent.\vspace{0mm}
\begin{enumerate}
\item $[\K_{\R'(I^n)}]_{\ge -b-1}\cong\R(I^n)(b+1)$ as a graded $\R(I^n)$-module.\vspace{2mm}
\item $[\K_{\R'(I)}]_{\ge n(-b-1)}\cong\R(\F^\ast)(n(b+1))$ as a graded $\R(I)$-module.
\vspace{2mm}
\item $[\K_{\R'(\F^\ast)}]_{\ge n(-b-1)}\cong\R(\F^\ast)(n(b+1))$ as a graded $\R(\F^\ast)$-module.
\end{enumerate}
\end{prop}

\begin{proof}
The implication $(1)\Rightarrow (3)$ directly follows from Lemma \ref{rr} (2). The implication  $(3)\Rightarrow (2)$ follows form $\K_{\R'(I)}\cong \K_{\R'(\F^\ast)}$ as a  graded $\R'(I)$-module (see Lemma \ref{ast} (2)). Let us check the implication $(2)\Rightarrow (1)$. Taking $[-]^{(n)}$ of the both side of that equivalence in the condition (2), we have $\big([\K_{\R'(I)}]_{\ge n(-b-1)}\big)^{(n)}\cong\big(\R(\F^\ast)(n(b+1))\big)^{(n)}$ as a graded $\R(I)^{(n)}$-module, so that $[\K_{\R'(I^n)}]_{\ge -b-1}\cong\R((\F^\ast)^{(n)})(b+1)$ as a graded $\R(I^n)$-module. Thus, the condition (1) holds, as $(\F^\ast)^{(n)}=\{I^{ni}\}_{i\in\Z}$.
\end{proof}

Herrmann, Hyry, and Ribbe reveals that, when the Rees algebra $\R(I^n)$ is Gorenstein, there exists
an embedding of $\G(\F^\ast)(-n-1)$ into the canonical module $\K_{\G(I)}$ of graded $\G(I)$-modules (see \cite[Theorem 3.9 (3)]{HHR}).
Let us notice the following result, which includes a slight generalization of their result.

\begin{lem}\label{congG}
Let $n$ be a positive integer and assume that $[\K_{\R'(I^n)}]_{\ge -b-1}\cong\R(I^n)(b+1)$ as a graded $\R(I^n)$-module for some integer $b$. 
Then
$$[\G(\w)]_{\ge n(-b-1)}\cong\G(\F^\ast)(n(b+1))$$ as a graded $\G(\F^\ast)$-module.
Hence, the following two assertions hold.\vspace{0mm}
\begin{enumerate}[$(1)$]
\item There is a monomorphism $\G(\F^\ast)(n(b+1)-1)\hookrightarrow \K_{\G(I)}$ of graded $\G(I)$-modules.\vspace{1mm}
\item There is a monomorphism $\G(\F^\ast)(n(b+1)-1)\hookrightarrow \K_{\G(\F^\ast)}$ of graded $\G(\F^\ast)$-modules. 
\end{enumerate}
\end{lem}

\begin{proof} 
Lemma \ref{ast} (2) implies that $\K_{\R'(\F^\ast)}\cong\R' (\w)$ as a graded $\R(\F^\ast)$-module, and hence it follows from Lemma \ref{rr} (2) that $\R' (\w)_{\ge n(-b-1)}\cong\R(\F^\ast)(n(b+1))$ as a graded $\R(\F^\ast)$-module.
Thus, we obtain that $\G(\w)_{\ge n(-b-1)}\cong\G(\F^\ast)(n(b+1))$ as a graded $\G(\F^\ast)$-module by Proposition \ref{embed2}.

(1) There exists a monomorphism $\G(\w)(-1)\hookrightarrow \K_{\G(I)}$ of graded $\G(I)$-modules by Lemma \ref{cf}. Then the required monomorphism follows from the isomorphism above.

(2) By Lemma \ref{ast}, $\w$ is also an extended canonical $\F^\ast$-filtration of $\K_A$, so that there exists a monomorphism $\G(\w)(-1)\hookrightarrow \K_{\G(\F^\ast)}$ of graded $\G(\F^\ast)$-modules by Lemma \ref{cf}. Then the required monomorphism follows from the isomorphism above.
\end{proof}

As a consequence of Lemma \ref{congG} (1), let us introduce the following result given by \cite[Theorem 3.9 and Corollary 3.10]{HHR}. Although their approach assumes that $\H_\fkM^d(\R(I))=0$, their proof is valid without this assumption for Remark \ref{injast} stated below. It should be noted that, under this assumption, their result \cite[Theorem 3.9 (3)]{HHR} not only establishes the monomorphism below but also determines the cokernel of its monomorphism.

\begin{rem}[\cite{HHR}]\label{injast}
Assume that $\R(I^n)$ is a Gorenstein ring for some positive integer $n$. Then there exists a monomorphism 
$
\G(\F^\ast)(-n-1)\hookrightarrow \K_{\G(I)}
$
 of graded $\G(I)$-modules. Hence, $\mathrm{a}(\G(I))\ge -n-1$. 
\end{rem}

\begin{proof}
It follows from Theorem \ref{main1} that $[\K_{\R'(I^n)}]_{\ge 1}\cong\R(I^n)(-1)$ as a graded $\R(I^n)$-module. Hence, the required monomorphism follows from Lemma \ref{congG} in the case where $b=-2$. 
The last assertion follows from $[\K_{\G(I)}]_{n+1}\neq 0$, as $A/I^\ast \neq 0$.
\end{proof}

We are now ready to prove Theorem \ref{in}.

\begin{proof}[Proof of Theorem \ref{in}]
According to Theorem \ref{main1}, 
$\R(I^n)$ is a Gorenstein ring if and only if  
$[\K_{\R'(I^n)}]_{\ge 1}\cong\R(I^n)(-1)$ as a graded $\R(I^n)$-module. 
Therefore, applying Proposition \ref{vero} in the case where $b=-2$, we directly obtain the required equivalence. The last assertion follows from Lemma \ref{congG} (2).
\end{proof}

Let us now present a result of the relationship between $\ai '(I^n)$ and $\ai '(I)$ as follows.

\begin{lem}\label{equal}
Let $n$ be a positive integer and assume that $[\K_{\R'(I^n)}]_{\ge -b-1}\cong\R(I^n)(b+1)$ as a graded $\R(I^n)$-module for some integer $b$. Then the following four conditions are equivalent.
\begin{enumerate}
\item $\ai '(I^n)= b$.\vspace{1mm}
\item $\mathrm{a'}(I)= n(b+1)-1$.
\vspace{1mm}
\item $\K_{\R'(I)}\cong\R'(\F^\ast)(n(b+1))$ as a graded $\R(I)$-module.
\vspace{1mm}
\item $\K_{\R'(\F^\ast)}\cong\R'(\F^\ast)(n(b+1))$ as a graded $\R(\F^\ast)$-module.\end{enumerate}
Hence, if $b=\ai(\G(I^n))$, then the equality $\mathrm{a'}(I)= n(\ai(\G(I^n))+1)-1$ holds.
\end{lem}

\begin{proof}
$(1)\Rightarrow (2)$ Suppose that $\ai '(I^n)= b$. By Lemma \ref{rr} (4), we have the inequality $\mathrm{a'}(I)+1\ge n(\ai '(I^n)+1)$. Proposition \ref{formula} implies the converse inequality holds true, so that $\mathrm{a'}(I)= n(b+1)-1$, as required.

$(2)\Rightarrow (4)$ We can take $y\in\K_A$ such that $0:_Ay=0$ and $\w_{i}=(I^{i+n(b+1)})^\ast y$ for all integers $i\ge n(-b-1)$ by Lemma \ref{rr}, in particular $\w_{n(-b-1)}=Ay$. By the definition of $\ai'(I)$, we have $\K_A=\w_{-\ai'(I)-1}$.
The condition (2) implies that 
$\w_{-\ai'(I)-1}=\w_{n(-b-1)}$, and therefore $\K_A=Ay$. Let $f: A\to \K_A=Ay$ be the $A$-isomorphism given by $1\mapsto y$, and then $f((I^{i+n(b+1)})^\ast)=\w_i$ for every $i\in\Z$.
Hence, setting $\varphi: \R'\big(\F^\ast(n(b+1))\big) \xrightarrow{}\R' (\w)$ to be the graded $\R'(\F^\ast)$-homomorphism extended from $f$, we observe $\varphi$ is bijective.
Thus,
$\K_{\R'(\F^\ast)}\cong\R'(\F^\ast)(n(b+1))$ as a graded $\R'(\F^\ast)$-module, as desired.

$(4)\Rightarrow (3)$ This follows form $\K_{\R'(I)}\cong \K_{\R'(\F^\ast)}$ as a  graded $\R'(I)$-module. 

$(3)\Rightarrow (1)$ 
From the condition (3), we obtain that
$$\K_{\R'(I)}\big/t^{-1}\K_{\R'(I)}\cong \left(\R'(\F^\ast)\big/t^{-1}\R'(\F^\ast)\right)(n(b+1))=\G(\F^\ast)(n(b+1))$$ as a graded $\G(I)$-module, which implies that $\mathrm{a'}(I)= n(b+1)-1$, as $I^\ast\neq A$.

Let us check the last assertion. We have $\ai (\G(I^n))\ge \ai '(I^n)\ge b$ by Lemma \ref{-a-1} and Lemma \ref{rr} (3). 
Since $\ai (\G(I^n))= b$, we get $\ai '(I^n)= b$, and thus $\mathrm{a'}(I)+1= n(b+1)$ by the implication $(1)\Rightarrow (2)$, as desired.
\end{proof}

The following theorem
elucidates the relationship between a quasi-Gorenstein properties $\R'(I^n)$ and $\R'(\F^\ast)$, assuming the ring $\R(I^n)$ to be Cohen-Macaulay, without requiring the ring $\R(I)$ to be Cohen-Macaulay. Moreover, 
its last assertion confirms that the quasi-Gorenstein properties of the ring $\R'(I^n)$ determine the $a$-invariant of the ring $\G(I)$, provided $\R(I^n)$ is a Cohen-Macaulay ring.

\begin{thm}\label{qGR'b}
Let $b$ be an integer. Assume that $\R(I^n)$ is a Cohen-Macaulay ring for some integer $n>0$. Then the following three conditions are equivalent.
\begin{enumerate}
\item $\K_{\R'(I^n)}\cong\R'(I^n)(b+1)$ as a graded $\R'(I^n)$-module.\vspace{1mm}
\item $\K_{\R'(I)}\cong\R'(\F^\ast)(n(b+1))$ as a graded $\R'(I)$-module.
\vspace{1mm}\item $\K_{\R'(\F^\ast)}\cong\R(\F^\ast)(n(b+1))$ as a graded $\R(\F^\ast)$-module.
\end{enumerate}
When this is the case, the equalities $b=\ai(\G(I^n))$ and
$$\mathrm{a}(\G(I))= n(b+1)-1$$
hold true.
\end{thm}

\begin{proof}
$(1)\Rightarrow (2)$ 
The condition (1) implies that $[\K_{\R'(I^n)}]_{\ge -b-1}\cong\R(I^n)(b+1)$ as a graded $\R(I^n)$-module and that
$\left(\K_{\R'(I^n)}\big/t^{-1}\K_{\R'(I^n)}\right)(-1)\cong \left(\R'(I^n)\big/t^{-1}\R'(I^n)\right)(b)$ as a graded $\G(I^n)$-module. The latter isomorphism yields the equality $\ai'(I^n)=b$, as $I^n\neq A$, and thus, the condition (2) holds by Lemma \ref{equal}.

$(2)\Leftrightarrow (3)$ This directly follows from form Lemma \ref{equal}.

$(3)\Rightarrow (1)$ The condition (3) implies that $\big(\K_{\R'(I)}\big)^{(n)}\cong\big(\R'(\F^\ast)(n(b+1))\big)^{(n)}$ as a graded $\R'(I)^{(n)}$-module, so that $\K_{\R'(I^n)}\cong\R((\F^\ast)^{(n)})(b+1)$ as a graded $\R(I^n)$-module. Since $\R(I^n)$ is a Cohen-Macaulay ring, we have $(\F^\ast)^{(n)}=\{I^{ni}\}_{i\in\Z}$. Thus, the condition (1) holds.

Let us shows the last assertion. Since $\R(I^n)$ is a Cohen-Macaulay ring, we have 
$$\K_{\G(I^n)}\cong \left(\K_{\R'(I^n)}\big/t^{-1}\K_{\R'(I^n)}\right)(-1)$$ as a graded $\G(I^n)$-module by Corollary \ref{maincor}. Hence, we obtain that $$\K_{\G(I^n)}\cong \left(\R'(I^n)\big/t^{-1}\R'(I^n)\right)(b)$$ as a graded $\G(I^n)$-module because $\K_{\R'(I^n)}\cong\R'(I^n)(b+1)$ as a graded $\R(I^n)$-module. Since $A/I^n\neq 0$, we get the equality $b=\ai(\G(I^n))$, as desired. Then we obtain from the last assertion in Lemma \ref{equal} that $\mathrm{a}'(I)= n(\ai(\G(I^n)+1)-1$, whence 
$$\ai(\G(I^n))=\big((\mathrm{a'}(I)+1)/n\big)-1.$$ 
The equality $\mathrm{a}(\G(I^n))=[\ai(\G(I))/n]$ follows from \cite[Corollary 4.3]{HZ}, where $[-]$ stands for the smallest integral part, so that $[\ai(\G(I))/n]=\big((\mathrm{a'}(I)+1)/n\big)-1$. Hence, it follows from the definition of $[-]$ that
$\ai(\G(I))/n<(\mathrm{a'}(I)+1)/n,$ whence $\ai(\G(I))\le\mathrm{a'}(I)$. From Lemma \ref{-a-1}, we obtain the equality $\ai(\G(I))=\mathrm{a'}(I)$.
This yields the required equality $\mathrm{a}(\G(I))= n(b+1)-1$, as $b=\ai(\G(I^n))$ and $\ai(\G(I^n))=\big((\mathrm{a'}(I)+1)/n\big)-1$.
\end{proof}

Let us here give a proof of Corollary \ref{incor} as follows.

\begin{proof}[Proof of Corollary \ref{incor}]
According to Corollary \ref{maincor1}, the condition (1) is equivalent to saying that $\K_{\R'(I^n)}\cong\R'(I^n)(-1)$ as a graded $\R'(I^n)$-module, as $\grade I\ge 2$. Applying Theorem \ref{qGR'b} for $b=-2$, we get the required equivalence and the last assertion. 
\end{proof}

Concluding this paper, let us prove the assertion $(3)$ of Example \ref{qGex}.

\begin{proof}[Proof of the assertion $(3)$ of Example \ref{qGex}]
$(\Rightarrow )$  
We have $\ai(\G(I))=-d+1$ by Example \ref{qGex} (2). Then the required equality follows from the last assertion in Corollary \ref{incor}. 

$(\Leftarrow )$ We obtain from Example \ref{qGex} (2) that $\K_{\R'(I^{d-2})}\cong\R'(I^{d-2})(-1)$ as a graded $\R'(I^{d-2})$-module. Therefore, it follows from \cite[Theorem 3.1]{K2} together with Corollary \ref{maincor1} that $\R(I^{d-2})$ is a Gorenstein ring. 
\end{proof}

\bibliographystyle{amsplain}

\end{document}